\documentclass[11pt]{article}
\usepackage{hyperref,etex,amsfonts,amssymb,amsmath,amsthm,amscd,euscript,
array,mathrsfs,multirow,booktabs}
\usepackage[title,titletoc,toc]{appendix}
\usepackage[all]{xy}

\input{mymacros-3.sty}
\author{Siddhartha Sahi\footnote{Department 
of Mathematics, Rutgers University, \texttt{sahi@math.rutgers.edu}.
}
\,, Hadi Salmasian\footnote{
Department of Mathematics and Statistics, University of Ottawa, 
\texttt{hsalmasi@uottawa.ca}.
}}

\begin{document}

\title{The Capelli problem for $\gl(m|n)$ and the spectrum of invariant
differential operators}

\maketitle

\begin{abstract}
The ``Capelli problem''
for the symmetric pairs
$(\gl\times \gl,\gl)$
$(\gl,\g{o})$, and 
$(\gl,\g{sp})$ is closely related to the theory of
 Jack polynomials and shifted Jack polynomials for  
 special values of the parameter (see \cite{KostantSahi1}, \cite{KostantSahi2}, \cite{KnopSahi},
\cite{Okounkov}).  In this paper, we extend this connection to the Lie superalgebra setting, namely to the 
supersymmetric pairs 
$(\g g,\g k):=(\gl(m|2n),\g{osp}(m|2n))$ and 
$(\gl(m|n)\times\gl(m|n),\gl(m|n))$, acting on
$W:=\sS^2(\C^{m|2n})$ and $\C^{m|n}\otimes(\C^{m|n})^*$.

To achieve this goal, we first prove that the center of the universal 
enveloping algebra of the Lie superalgebra $\g g$
maps surjectively onto the algebra $\sPD(W)^{\g g}_{}$ of $\g g$-invariant differential
operators on the superspace $W$, thereby
providing an affirmative answer to the ``abstract'' Capelli problem for $W$. Our
proof works more generally for $\gl(m|n)$ 
acting on $\sS^2(\C^{m|n})$ and
is new even for the ``ordinary'' cases  
($m=0$ or $n=0$)
considered by
Howe and Umeda in \cite{HoweUmeda}.

We next describe a natural basis $\left\{ D_{\lambda }\right\} $ of 
$\sPD(W)^{\g g}$, that we call the Capelli basis. Using the above result on the abstract Capelli problem, we
generalize the work of Kostant and Sahi \cite{KostantSahi1}, \cite{KostantSahi2}, \cite{Sahi} by showing that the spectrum of 
$D_{\lambda }$ is given by a polynomial $c_{\lambda }$, which is
characterized uniquely by certain vanishing and symmetry properties.

We further show that the top homogeneous parts of the eigenvalue polynomials 
$c_{\lambda }$ coincide with the spherical polynomials $d_{\lambda }$, which
arise as radial parts of $\g k$-spherical vectors of finite
dimensional $\g g$-modules, and which are super-analogues of
Jack polynomials. This generalizes results of Knop and Sahi \cite{KnopSahi}.

Finally, we make a precise connection between the  polynomials $c_\lambda$ and
the shifted super Jack polynomials of Sergeev and Veselov \cite{SerVes} for special values of the parameter. We show
that the two families are related by a change of coordinates that we
call the \textquotedblleft Frobenius transform\textquotedblright. 
%For the second case we show that similar results hold for %the
%action of  $\gl(V)\times \gl(V)$ on $V\otimes V^*$, where %$V:=\C^{m|n}$, and  can  be deduced by
%combining 
%results of Molev \cite{Molev} with \cite{SerVes} for %$\theta=1$. 
\\[1.5mm]

\end{abstract}

\section{Introduction}

One of the most celebrated results in classical invariant theory is the
Capelli identity for $n\times n$ matrices. It plays a fundamental role in
Hermann Weyl's book \cite{Weyl}. The Capelli identity is profoundly connected to the 
representation theory of the Lie algebra $\gl(n)$.
The well-known article of Roger Howe 
\cite{HoweRem}
elucidated this
connection by giving a conceptual proof of the Capelli identity.

Later on, Howe and Umeda \cite{HoweUmeda} 
generalized the Capelli identity to the
setting of multiplicity-free spaces by posing and solving two general
questions, which they called the \emph{abstract} and \emph{concrete} Capelli
problems. 
Let $\g g$ be a complex reductive Lie algebra, and $W$ be a
multiplicity-free $\g g$-space.
 The abstract Capelli problem asks
whether the centre $\bfZ(\g g)$ of $\bfU(\g g)$, the universal enveloping
algebra of $\g g$, maps surjectively onto the algebra $\sPD(W)^{\g g}$ of $\g g$%
-invariant polynomial-coefficient differential operators on $W$. The
concrete Capelli problem asks for explicit elements of $\bfZ(\g g)$ whose
images generate $\sPD(W)^\g g$.

Around the same time as \cite{HoweUmeda}, Kostant and Sahi \cite{KostantSahi1}, \cite{KostantSahi2} considered a slightly
different question, which we shall refer to here as the Capelli \emph{eigenvalue} problem. It turns out that the algebra $\sPD(W)^{\g g}$ admits a
natural basis $D_{\lambda }$ -- the \emph{Capelli basis} -- which is indexed
by the monoid of highest weights of $\g g$-modules $V_{\lambda }$
occurring in the symmetric algebra $\sS\left( W\right) $. The Capelli
eignevalue problem asks for determination of the eigenvalue of $D_{\mu }$
on $V_{\lambda }$. It turns out that these eignvalues are of the form $\varphi
_{\mu }\left( \lambda +\rho \right) $, where  $\varphi _{\mu }$ is a
symmetric polynomial and $\rho $ is a certain ``rho-shift". Although \cite{KostantSahi1},
\cite{KostantSahi2} are
written in the context of symmetric spaces, similar ideas work for the
multiplicity-free setting, see \cite{Knop}.

In \cite{Sahi}, it is shown that the polynomials $\varphi _{\mu }$ are uniquely
characterized by certain vanishing conditions. In fact \cite{Sahi} considers a
general class of polynomials, depending on several parameters, and in
\cite{KnopSahi} it is shown that a one-parameter subfamily of these polynomials is
closely related to \emph{Jack polynomials}. More precisely, the Knop--Sahi
polynomials are inhomogeneous polynomials, but their top degree terms are the
(homogenenous) Jack polynomials. In the special case where the 
value of the parameter $\theta$
corresponds to a symmetric space, the Knop--Sahi polynomials are the
eigenvalue polynomials $\varphi _{\lambda }$, and the Jack polynomials are the
spherical polynomials in $V_{\lambda }$.
Many properties of the  Knop--Sahi polynomials were subsequently proved by Okounkov and Olshanski
\cite{Okounkov}, who worked with a slight modification of these polynomials which they call \emph{shifted Jack polynomials}.

The goal of our paper is to extend this circle of ideas to the Lie superalgebra setting. 
The Jack polynomials for
$\theta=1,\frac{1}{2},2$ 
are spherical polynomials of the symmetric pairs 
\[
(\gl(n)\times \gl(n),\gl(n)),\
(\gl(n),\g{o}(n)),\
(\gl(2n),\g{sp}(2n)).
\]
In the Lie superalgebra setting the last two  come together and  there are only two pairs to consider, namely 
\[
(\gl(m|n)\times \gl(m|n),\gl(m|n))\text{ and }(\gl(m|2n),\g{osp}(m|2n)).
\]
The extension of the theory to the first case is easier and, as we show in 
Appendix \ref{appxB},  can
be achieved by combining the work of  Molev 
\cite{Molev} with that of Sergeev--Veselov 
\cite{SerVes}.
Therefore in the body of the paper 
we concentrate on the second case, which corresponds 
to the  action of $\gl(m|2n)$ on $\sS^2(\C^{m|2n})$. 

%We consider the action of the Lie superalgebra $\gl(m|n)$ on the symmetric square $\sS^2(\C^{m|n})$. 
In the extension of the aforementioned theory to Lie superalgebras, at least two serious difficulties arise. The first issue is that complete reducibility of finite dimensional representations 
fails in the case of Lie superalgebras. 
Complete reducibility is crucial in the approach to the abstract Capelli problem in \cite{HoweUmeda} and \cite{GoodmanWallach},
where it is needed to split off the $\g g$-invariants in $\bfU(\g g)$ and $\sPD(W)^{\g g}$.

The second issue that arises is that the 
Cartan--Helgason Theorem is not known in full generality for Lie superalgebras. In particular, it is not immediately obvious that every $\gl(m|2n)$-module that appears in $\sP(\sS^2(\C^{m|2n}))$ has an 
$\g{osp}(m|2n)$-invariant vector. Furthermore, in the purely even case the fact that a spherical vector is determined uniquely by its radial component  follows from the $KAK$ decomposition, which does not have an analogue in the context of supergroups.  These issues complicate the formulation and proof of  
Theorem \ref{prpgw} and Theorem \ref{MAINTHM}.

To achieve our goal in this paper, we have to overcome  the above difficulties.
Our first main result is  Theorem 
\ref{prpgw}, where we give an affirmative answer to the abstract Capelli problem of Howe
and Umeda for the Lie superalgebra $\gl(m|n)$ acting on 
$\sS^2(\C^{m|n})$. In fact in Theorem \ref{prpgw} we prove a slightly more
precise statement that a $\gl(m|n)$-invariant differential operator of order 
$d$ is in the image of an element of $\bfZ(\gl(m|n))$ which has order $d$
with respect to the standard filtration of $\bfU(\gl(m|n))$. The latter
refinement is needed in the proofs of Theorem \ref{MAINTHM} and Theorem \ref{thmconnSV}, 
which we will elaborate on below.
For the superpair $(\gl(m|n)\times\gl(m|n),\gl(m|n))$,  
the corresponding abstract Capelli theorem indeed follows 
directly from the work of Molev \cite{Molev}, as explained in Appendix \ref{appxB} (see Theorem \ref{thabsscap}).

Our second main goal  concerns
the Capelli eigenvalue problem for the Lie superalgebra  $\gl(m|2n)$ acting on $\sS^2(\C^{m|2n})$.
Extending \cite[Theorem 1]{Sahi}, we show 
that the eigenvalues of the Capelli operators are given by polynomials 
$c_\lambda$
given in Definition \ref{dfcl}  
that are characterized by suitable symmetry and  vanishing conditions
(see Theorem \ref{thm-unqclam}).
Furthermore, in Theorem %
\ref{MAINTHM} we show that the 
top degree homogeneous term of the eigenvalue polynomial 
$c_\lambda$  is equal to the spherical polynomial $d_\lambda$ given in Definition \ref{dfbr}. 
This extends the result proved by Knop and Sahi 
in \cite{KnopSahi}
explained  above.
The corresponding result for the superpair $(\gl(m|n)\times \gl(m|n),\gl(m|n))$ is Theorem \ref{THMAppB2}. In this case the spherical and eigenvalue polynomials turn out to be the
well known \emph{supersymmetric Schur polynomials} and \emph{shifted supersymmetric Schur} polynomials \cite{Molev}.

Our third main goal is to establish a precise relation between our
eigenvalue polynomials $c_\lambda$ and certain polynomials called
\emph{shifted super Jack
polynomials}. 
In
connection with the eigenstates of the deformed Calogero--Moser--Sutherland
operators
\cite{SerVes}, Sergeev and Veselov define a family of $(m+n)$-variable
polynomials $SP^*_\flat$,  parametrized by $(m,n)$-hook partitions $\flat$
(see Definition \ref{dfnhookprn}). 
In the case of the superpair $(\gl(m|2n),\g{osp}(m|2n))$,
in Proposition \ref{prpQlam} and Theorem \ref{thmconnSV}, we
prove that the $c_\lambda$ and the $SP^*_\flat$ are related by the \emph{Frobenius transform}, see  Definition \ref{FrobT}. 
The corresponding statement for the pair
$(\gl(m|n)\times\gl(m|n),\gl(m|n))$ is Theorem 
\ref{THMAB3}.

It is worth
mentioning that in \cite{SerVes}, the shifted super Jack polynomials 
are obtained as
the image of the shifted Jack polynomials under a certain \emph{shifted Kerov map}, and the fact that the image of the shifted Kerov map is a
polynomial is indeed a nontrivial statement which is proved in \cite{SerVes} indirectly. However, our definition of $c_\lambda$ is
more conceptual, and the proofs are more straightforward, as they are based
on the Harish--Chandra homomorphism and the solution of the abstract Capelli problem.
Furthermore, the work of Sergeev and Veselov does not address the relation
with spherical representations of Lie superalgebras. 
Our paper establishes this connection.

We now outline the structure of this article. Section \ref{section1} defines
the basic notation that is used throughout the rest of the paper. The
solution to the abstract Capelli problem for $\gl(m|n)$ 
acting on $\sS^2(\C^{m|n})$ 
is given in Section \ref%
{prfof571}. In Section \ref{Secsuperpr} we study spherical highest weight
modules of $\gl(m|2n)$. We prove in Proposition \ref{le-dlam} and Remark \ref{rmk-v*} 
that every irreducible $\gl(m|2n)$-submodule of $\sS(\sS^2(\C^{m|2n}))$ or $\sP(\sS^2(\C^{m|2n}))$
has a
unique (up to scalar) nonzero $\g{osp}(m|2n)$-fixed vector. It is worth mentioning that  Proposition \ref{le-dlam} does not follow from
the work of Alldridge and Schmittner \cite{AllSch}, since they need to
assume that the highest weight is ``high enough'' in some sense. In Section 
\ref{Sec-Sec5},
we prove Theorem \ref{MAINTHM}
and Theorem
\ref{thm-unqclam} (see the second goal above). 
Section 
\ref{SecRelSer} is devoted to connecting the eigenvalue polynomials $c_\lambda$ to the 
shifted super Jack polynomials of 
 Sergeev and Veselov \cite{SerVes}.
Appendix \ref{sec-pflem} contains the proof of 
Proposition \ref{DGBRVVV}. Finally, in Appendix B we outline the proofs of our main theorems for the case of 
$\gl(m|n)\times\gl(m|n)$ acting on $\C^{m|n}\otimes (\C^{m|n})^*$.

We now briefly describe the structure of our proofs.
There is no mystery in the formulation of the
statement of Theorem \ref{prpgw}, but its proof is \emph{not} a simple
generalization of any of the existing proofs in the purely even case (i.e.,
when $n=0$) that we are aware of. Our approach is inspired by the proof given
by Goodman and Wallach in 
\cite[Sec. 5.7.1]{GoodmanWallach}, but it diverges quickly because their argument
relies heavily on the complete reducibility of rational representations of a
reductive algebraic group. Our proof proceeds by induction, after we show that the symbol of an invariant differential operator is in the image of $\bfZ(\gl(m|n))$. This symbol is in the span of invariant tensors 
$\bft_\sigma$ defined in \eqref{tsigmeqqqqn} where  $\sigma$ is a permutation. The next step is to reduce the latter problem to the case where $\sigma$ is a cycle of consecutive letters. Finally, we prove that in this special case, $\bft_\sigma$ is in the image of 
the \emph{Gelfand elements} $\mathrm{str}(\mathbf E^d)\in\bfZ(\gl(m|n))$ defined 
in Lemma \ref{gelfzhel}.

The idea behind the proof of Theorem \ref{MAINTHM} is as follows. Let $D_\lambda$ be a Capelli operator, and let $z_\lambda\in\bfZ(\gl(m|2n))$ be an element in the inverse image of $D_\lambda$, whose existence  follows from Theorem \ref{prpgw}.  
We show that 
modulo the natural isomorphism induced by the trace form,
the spherical polynomial $d_\lambda$ is the diagonal restriction of the symbol of the 
polynomial-coefficient differential operator corresponding to the radial part of $z_\lambda$. Furthermore, we show that the eigenvalue polynomial $c_\lambda$ 
is the image of the  
 radial part of $z_\lambda$  (see  Lemma \ref{cmulZa}). We combine the latter two statements, as well as   the refinement of the abstract Capelli problem 
obtained in Theorem \ref{prpgw}, to prove Theorem \ref{MAINTHM}.

Finally, the proof of Theorem \ref{thmconnSV} goes as follows. 
By considering the action of $z_\lambda$ on the lowest weight vector of an irreducible $\gl(m|2n)$-module $V_\mu$, we prove 
in Proposition \ref{prpQlam} that $c_\lambda(\mu)$ is a polynomial in the highest weight $\mu^*$ of the contragredient representation $V_\mu^*$. Denoting the latter polynomial $c_\lambda^*$, we verify that after a \emph{Frobenius transform} (see Definition \ref{FrobT}),  the polynomial $c_\lambda^*$ satisfies the supersymmetry and vanishing properties of the shifted super Jack polynomials of \cite{SerVes}. It then follows that the two polynomials 
coincide up to a scalar multiple.

We now elaborate on some of the new techniques and ideas 
introduced in our paper. Our method of proof
of Theorem \ref{prpgw}
yields a recursive procedure for
expressing a given invariant differential operator explicitly as the image
of an element of $\bfZ(\gl(m|n))$ using Gelfand elements. In the setting of ordinary Lie algebras, 
Howe and Umeda \cite{HoweUmeda} obtain such a formula for the 
\emph{generators} of the algebra of invariant differential operators. By contrast, even 
in this setting, our construction 
is more general and gives explicit pre-images for a \emph{basis} of the algebra. 

Another new idea is the construction of spherical vectors in tensor representations of $\gl(m|2n)$ using symbols of the Capelli operators $D_\lambda$ (see Lemma \ref{lemZnZk=0}).\\
% This remedies the lack of the Cartan--Helgason theorem for symmetric superpairs. \\

\noindent\textbf{Acknowledgement.} 
We thank the referee for useful comments and pointing to references for Lemma \ref{gelfzhel}.
Throughout this project, we benefited from discussions with Alexander Alldridge, Ivan Dimitrov, Roe Goodman, Roger Howe, Friedrich Knop, James Lepowsky, Nolan Wallach, and Tilmann Wurzbacher. We thank them  for  fruitful and encouraging discussions. The second author is supported by an NSERC Discovery Grant.

\section{Notation}

\label{section1} We briefly review the basic theory of vector superspaces
and Lie superalgebras. For more detailed expositions, see for example \cite%
{ChWabook} or \cite{Musson}. Throughout this article, all vector spaces will
be over $\C$.
Let $\SVec$ be the symmetric monoidal category of $\Ztwo$-graded vector
spaces,
where $\Ztwo:=\left\{\eev,\ood\right\}$.
Objects of $\SVec$ are of the form
$U=U_\eev\oplus U_\ood$. The parity of a homogeneous vector $u\in U$ is denoted by $|u|\in\Ztwo$.  
For any two 
$\Ztwo$-graded vector spaces  $U$ and $U'$, the vector space 
$\Hom_\C
(U,U^{\prime})$ is naturally  $\Ztwo$-graded, and the morphisms of 
$\SVec$ are defined by
$\Mor_{\SVec}(U,U^{\prime
}):=\Hom_\C
(U,U^{\prime})_\eev$.
The
 symmetry isomorphism of $\SVec$ is defined by
\begin{equation}
\label{symisomdf}
\braid_{U,U^{\prime }}:U\otimes U^{\prime }\to U^{\prime }\otimes U \ ,\
u\otimes u^\prime 
\mapsto (-1)^{|u|\cdot |u^{\prime }|}u^{\prime }\otimes u.
\end{equation}
We remark that throughout this article, the defining relations which involve
parities of vectors should first be construed as  relations for homogeneous
vectors, and then be extended by linearity to arbitrary vectors.

The identity element of the associative superalgebra  
$
\mathrm{End}_\C(U):=%
\Hom_\C(U,U)
$ 
will be denoted by $1_U$. Note that $\mathrm{End}_\C(U)$ is
a Lie superalgebra with the standard commutator $[A,B]:=AB-(-1)^{|A|%
\cdot|B|}BA$. Set  $U^*:=\Hom_\C(U,\C^{1|0})$. The map 
\begin{equation}  \label{U'U*}
U^{\prime}\otimes U^* \to
\Hom_\C(U,U^{\prime }) \ ,\ 
u^{\prime}\otimes u^*\mapsto T_{u^{\prime}\otimes u^*},
\end{equation}
where $T_{u^{\prime}\otimes u^*}(u):=\lag u^*,u\rag u^{\prime }$ for all $u\in U$, is
an isomorphism in the category $\SVec$.

The natural representation (in the symmetric monoidal category $\SVec$) of
the symmetric group $S_d$ on $U^{\otimes d}$ is explicitly given by 
\begin{equation}  \label{dfTsigm}
\sigma\mapsto T_{U,d}^\sigma\ \,,\,\ T^{\sigma}_{U,d}( v_1\otimes \cdots
\otimes v_{d}) := (-1)^{\seps\left(\sigma^{-1};v_1,\ldots,v_{d}\right)}
v_{\sigma^{-1}(1)}\otimes\cdots\otimes v_{\sigma^{-1}(d)}\,,
\end{equation}
where 
\begin{equation}  \label{sepssign}
\seps(\sigma;v_1,\ldots,v_d):= \sum_{\substack{ 1\leq r<s\leq d \\ %
\sigma(r)>\sigma(s) }} |v_{\sigma(r)}|\cdot|v_{\sigma(s)}|.
\end{equation}
The supersymmetrization map $\ssym^d_U:U^{\otimes d}\to U^{\otimes d}$ is defined by
\begin{equation}
\label{symdUeq}
\ssym^d_U :=\frac{1}{d!} \sum_{\sigma\in S_d} T^\sigma_{U,d}.
\end{equation}
 We set 
\[
\sS^d(U):=\ssym^d_U\left(U^{\otimes d}\right)
,\  
\sS(U):=\bigoplus_{d\ge 0}\sS
^d(U),\
\sP(U):=\sS(U^*), 
\text{ and }
\sP^d(U):=\sS^d( U^*).
\] 
%Recall that $\sS(U)\cong\sS(U_\eev)\otimes \Lambda(U_\ood)$.
%\cong\sS(U_\eev^*)\otimes \Lambda(U_\ood^*)
Every $\eta\in U^*_\eev$ extends canonically to a homomorphism of superalgebras 
\begin{equation}  \label{dfheta}
\sfh_\eta:\sS(U)\to \C\cong \C^{1|0}
\end{equation}
defined by $\sfh_\eta \big(
\ssym^d_U(u_1\otimes \cdots \otimes u_d) \big)
:=\eta(u_1)\cdots \eta(u_d) $ for $d\geq 0$ and $u_1,\ldots,u_d\in U$.

If $\g g$ is a Lie superalgebra, 
then  $\bfU(\g g)$
denotes the universal enveloping algebra of $\g g$, and 
\[
\C=\bfU^0(\g g)\subset
\bfU^1(\g g)\subset\cdots
\subset
\bfU^d(\g g)\subset\cdots
\]
denotes the standard filtration of $\bfU(\g g)$.
If $(\pi,U)$ is a $\g g$-module, we define 
\begin{equation*}
U^\g g:=\{u\in U\ :\ \pi(x)u=0\text{ for every }x\in\g g\}. 
\end{equation*}
If $(\pi^{\prime },U^{\prime })$ is another $\g g$-module, then $\Hom_\C%
(U,U^{\prime })$ is also a $\g g$-module, with the action 
\begin{equation}  \label{xcdtTT}
(x\cdot T)u:=\pi^{\prime}(x)Tu-(-1)^{|T||x|}T\pi(x)u\,\ 
\end{equation}
for $x\in\g g$, $T\in\Hom_\C%
(U,U^{\prime })$, and $u\in U$.
The special case $U^*:=\Hom_\C(U,\C^{1|0})$, where $\C^{1|0}$ is the trivial $%
\g g$-module, is the \emph{contragredient} $\g g$-module. Moreover,
the map \eqref{U'U*} %and the canonical isomorphism
% $U^*\otimes U'^*\cong (U'\otimes U)^*$ are 
is $\g g$-equivariant. Set

\begin{equation*}
\Hom_\g g(U,U^{\prime }):= \Hom_\C(U,U^{\prime})^{\g g} \ \text{ and }\ \mathrm{%
End}_\g g(U):=\mathrm{End}_\C(U)^\g g. 
\end{equation*}
The category of $\g g$-modules is a symmetric monoidal category, and 
\[
\Mor_{\g g\text{-mod}}(U,U^{\prime })\cong \Hom_\g g(U,U^{\prime })_\eev
.\]

Fix a $\Ztwo$-graded vector space $W$. For every homogeneous $w\in W$, let $\partial_w$
be the superderivation of $\sP(W)$ with parity $|w|$ that is defined
uniquely by 
\begin{equation}  \label{pww*}
\partial_w(w^*):=(-1)^{|w|}\lag w^*,w\rag\text { for every } w^*\in W^*\cong %
\sP^1(W).
\end{equation}
Thus, $\partial_w(a_1a_2)=(\partial_w a_1)a_2+(-1)^{|w|\cdot|a_1|}a_1\partial_wa_2$ for
every $a_1,a_2\in\sP(W)$. 
For every $b\in \sS(W)$ we define $\partial_b\in%
\mathrm{End}_\C(\sP(W))$ as follows. First we set 
\begin{equation*}
\partial_{w_1\cdots w_r}:=\partial_{w_1}\cdots\partial_{w_r} \text{ for homogeneous }%
w_1,\ldots,w_r\in W, 
\end{equation*}
and then we extend the definition of $\partial_{b}$  to all $b\in \sS(W)$ by linearity.

Let $\sPD(W)$ be the associative superalgebra of polynomial-coefficient
differential operators on $W$. More explicitly, $\sPD(W)$ is the subalgebra
of $\mathrm{End}_\C(\sP(W))$ spanned by elements of the form $a\partial_b$,
where $a\in \sP(W)$ and $b\in\sS(W)$. If $W$ is a $\g g$-module, then $\sPD%
(W)$ is a $\g g$-invariant subspace of $\mathrm{End}_\C(\sP(W))$ with
respect to the $\g g$-action on $\mathrm{End}_\C(\sP(W))$ defined in %
\eqref{xcdtTT}. Furthermore, the map 
\begin{equation}  \label{eq-sfm:}
\sfm:\sP(W)\otimes \sS(W)\to \sPD(W)\ ,\ a\otimes b\mapsto a\partial_b
\end{equation}
is an isomorphism of $\g g$-modules (but not of superalgebras). The superalgebra $\sPD(W)$
has a natural filtration given by 
\begin{equation*}
\sPD^d(W):=\sfm\left(\sP(W)\otimes\left(\bigoplus_{r=0}^d \sS%
^r(W)\right)\right)\text{ for }d\geq 0. 
\end{equation*}
For $D\in\sPD(W)$, we write $\mathrm{ord}(D)=d$ if  $%
D\in\sPD^d(W)$ but $D\not\in\sPD^{d-1}(W)$. For every $d\geq 0$, the $d$-th
order symbol map 
\begin{equation*}
\widehat\sfs_d:\sPD^{d}(W)\to \sP(W)\otimes \sS^d(W) 
\end{equation*}
is defined by 
\begin{equation}  \label{eq-shmphh}
\widehat\sfs_d\left(\sPD^{d-1}(W)\right)=0\text{\, and\ \,} \widehat\sfs_d(%
\sfm(a))=a\ \text{ for }\ a\in \sP(W)\otimes \sS^d(W).
\end{equation}
If $D_r\in\sPD^{d_r}(W)$ for $1\leq r\leq k$, then 
\begin{equation}  \label{whsD1Dk}
\widehat\sfs_{d_1+\cdots+d_k}(D_1\ldots D_k)= \widehat\sfs_{d_1}(D_1)\cdots
\widehat\sfs_{d_k}(D_k),
\end{equation}
where the product on the right hand side takes place in the superalgebra 
$\sP(W)\otimes\sS(W)$.

\section{The abstract Capelli problem for $W:=\sS^2(V)$}

\label{prfof571}

Fix $V:=\C^{m|n}$ and set $\g g:=\gl(V):=\gl(m|n)$. We fix bases $\bfe%
_1,\ldots,\bfe_m$ for $V_\eev=\C^{m|0}$ and $\bfe_{\oline 1},\ldots,\bfe_{%
\oline n} $ for $V_\ood=\C^{0|n}$. Throughout this article we will use the
index set 
\begin{equation*}
\mathcal{I}_{m,n}:=\{1,\ldots,m,\oline 1,\ldots,\oline n\}. 
\end{equation*}
We define the parities of elements of $\mathcal{I}_{m,n}$ by 
\begin{equation*}
\left|i\right|: =\left|\bfe_i\right |= 
\begin{cases}
\eev & \text{ for }i\in\{1,\ldots,m\}, \\ 
\ood & \text{ for } i\in\{\oline 1,\ldots,\oline n\}.%
\end{cases}
\end{equation*}

For an associative superalgebra 
 $\cA=\cA_\eev\oplus\cA_\ood$, 
let $\mathrm{Mat}_{m,n}(\cA)$ be
the associative superalgebra of $(m+n)\times (m+n)$ matrices with entries in 
$\cA$, endowed with the $\Z_2$-grading obtained by the $(m,n)$-block form of
its elements. More precisely, $\mathrm{Mat}_{m,n}(\cA)_\eev$ consists of
matrices in $(m,n)$-block form 
\begin{equation}  \label{blockformgg}
\begin{bmatrix}
A & B \\ 
C & D%
\end{bmatrix}%
\end{equation}
such that the entries of the $m\times m$ matrix $A$ and the $n\times n$
matrix $D$ belong to $\cA_\eev$, whereas the entries of the matrices $B$ and 
$C$ belong to $\cA_\ood$. Similarly, $\mathrm{Mat}_{m,n}(\cA)_\ood$ consists
of matrices of the form \eqref{blockformgg} such that the entries of $A$ and 
$D$ belong to $\cA_\ood$, whereas the entries of $B$ and $C$ belong to $\cA_%
\eev$.  The supertrace of an element of $\mathrm{Mat}_{m,n}(\cA)$ is defined
by 
\begin{equation*}
\str\left( 
\begin{bmatrix}
A & B \\ 
C & D%
\end{bmatrix}
\right)=\mathrm{tr}(A)-\mathrm{tr}(D). 
\end{equation*}

Using the basis $\left\{\bfe_i\,:\,i\in\mathcal{I}_{m,n}\right\}$, we can
represent elements of $\g g$ by  $(m+n)\times (m+n)$ matrices, with rows
and columns indexed by elements of $\mathcal{I}_{m,n}$. For every $i,j\in%
\mathcal{I}_{m,n}$, let $E_{i,j}$ denote the element of $\g g$ corresponding
to the matrix with a 1 in the $(i,j)$-entry and 0's elsewhere. The standard
Cartan subalgebra of $\g g$ is 
\begin{equation*}
\g h:=\spn_\C\{E_{1,1},\ldots,E_{m,m},E_{\oline 1,\oline 1},\ldots, E_{%
\oline n,\oline n}\}, 
\end{equation*}
and the standard characters of $\g h$ are $\eps_1,\ldots,\eps_m,\eps_{\oline %
1},\ldots,\eps_{\oline n}\in\g h^*$, where $\eps_i(E_{j,j})=\delta_{i,j} $
for every $i,j\in\mathcal{I}_{m,n} $. The standard root system of $\g g$ is $%
\Phi:=\Phi^+\cup \Phi^-$ where 
\begin{equation*}
\Phi^+:=\Big\{
\eps_k^{}-\eps_l^{} \Big\}_{1\leq k<l\leq m} \cup \left\{ \eps_k^{}-\eps_{%
\oline l}^{} \right\} _{ 1\leq k\leq m,\,1\leq l\leq n} \cup \left\{ \eps_{%
\oline k}^{}-\eps_{\oline l}^{} \right\}_{1\leq k<l\leq n} 
\end{equation*}
and $\Phi^-:=-\Phi^+$. 
For $\alpha\in\Phi$, we set $\g g_\alpha:=\{x\in\g g\,:\, [h,x]=\alpha(h)x\text{ for every }h\in\g h\}$.
Then  $\g g=\g n^-\oplus\g h\oplus\g n^+$,
where $\g n^\pm:=\bigoplus_{\alpha\in\Phi^\pm}\g g_\alpha$. Finally, the
supersymmetric bilinear form 
\begin{equation}  \label{kappa}
\kappa:\g g\times\g g\to\C\ ,\ \kappa(X,Y):=\mathrm{str}(XY)\text{ for }
X,Y\in\g g,
\end{equation}
is nondegenerate and $\g g$-invariant.

Set $W:=\sS^2(V)$. The standard representation of $\g g$ on $V$ gives rise to
canonically defined representations on $V^{\otimes d}$, $W$, $W^*$, $\sS(W)$%
, and $\sP(W)$. From now on, we denote the latter representations of $\g g$
on $\sS(W)$ and $\sP(W)$ by $\rho$ and $\check\rho$ respectively.

The goal of the rest of this section is to prove Theorem \ref{prpgw}, which gives an affirmative answer to the abstract Capelli problem for the $\g g$-module $W$. Set 
\begin{equation*}
\bfZ(\g g):=\{x\in\bfU(\g g)\,:\,[x,y]=0\text{ for every }y\in\bfU(\g g)\}, 
\end{equation*}
where $[\,\cdot,\cdot\,]$ denotes the standard superbracket of $\bfU(\g g)$. 
\begin{lem}
\label{gelfzhel}
Let $\mathbf E\in\mathrm{Mat}_{m,n}(\bfU(\g g))_\eev$ be the matrix with entries 
\[
\mathbf E_{i,j}
:=
(-1)^{|i|\cdot|j|}E_{j,i}
\text{ for }i,j\in\mathcal I_{m,n}.
\]
Then $\mathrm{str}(\mathbf E^d)\in\bfZ(\g g)$ for every $d\geq 1$.
\end{lem}
\begin{proof}
This lemma can be found for example in 
\cite{sergeev82} or \cite{scheunert}. 
For the reader's convenience, we outline a proof 
(in the case $n=0$ it
is due to 
\v{Z}helobenko 
\cite{Zhel}). \\

\noindent\textbf{Step 1.} Let $\mathbf Z=[z_{i,j}]_{i,j\in\mathcal I_{m,n}}^{}
$ be an element of $\mathrm{Mat}_{m,n}(\bfU(\g g))_\eev$ that satisfies
\begin{equation}
\label{xijEkl}
[z_{i,j},E_{k,l}]
=
(-1)^{|i|\cdot |j|+|l|\cdot |j|}
\delta_{i,k}z_{l,j}
-
(-1)^{|j|\cdot |k|+|j|}\delta_{l,j}z_{i,k}
\text{ for }
i,j,k,l\in\mathcal I_{m,n}.
\end{equation}
Set $z:=\mathrm{str}(\mathbf Z)=\sum_{i\in \mathcal I_{m,n}}(-1)^{|i|}
z_{i,i}$. Then $
\left[z,E_{k,l}\right]=0
$ 
for
$k,l\in\mathcal I_{m,n}$, and thus $z\in\bfZ(\g g)$.\\

\noindent\textbf{Step 2.}
Let $\mathbf Z=[z_{i,j}]_{i,j\in\mathcal I_{m,n}}^{}$ and $\mathbf Z'=[z'_{i,j}]_{i,j\in\mathcal I_{m,n}}^{}
$ be elements of 
$\mathrm{Mat}_{m,n}(\bfU(\g g))_\eev$
that satisfy
\eqref{xijEkl}. Set $\mathbf Z'':=\mathbf Z\mathbf Z'$, so that
$\mathbf Z''=[z''_{i,j}]_{i,j\in\mathcal I_{m,n}}^{}$ where
 $z''_{i,j}:=\sum_{r\in \mathcal I_{m,n}}z_{i,r}z'_{r,j}$.
 Then
\begin{align*}
\left[z''_{i,j},E_{k,l}\right]
&=
\sum_{r\in\mathcal I_{m,n}}z_{i,r}[z'_{r,j},E_{k,l}]
+
\sum_{r\in\mathcal I_{m,n}}
(-1)^{(|r|+|j|)(|k|+|l|)}[z_{i,r},E_{k,l}]z'_{r,j}\\
&=(-1)^{|j|\cdot|i|+|j|\cdot |l|}\delta_{i,k}z''_{l,j}
-
(-1)^{|j|\cdot |k|+|j|}\delta_{l,j}z''_{i,k},
\end{align*}
so that the entries of $\mathbf Z''$  
satisfy 
\eqref{xijEkl} as well.\\

\noindent\textbf{Step 3.}
Since \eqref{xijEkl}
holds for the entries of $\mathbf E$, 
induction on $d$ and Step 2 imply that \eqref{xijEkl} holds for the entries of $\mathbf E^d$ 
 for every  $d\geq 1$, and thus $\mathrm{str}(\mathbf E^d)\in \bfZ(\g g)$ by Step 1. 
\end{proof}

Let $\left\{\bfe^*_i\ :\ i\in\mathcal{I}_{m,n}\right\}$ be the basis of $V^*$
dual to the basis $\left\{\bfe_i\ :\ i\in\mathcal{I}_{m,n}\right\}$. Set
%$i,j\in\mathcal I_{m,n}$ define
\begin{equation}  \label{eq-dfxijyij}
x_{i,j}:=\frac{1}{2}\left(\bfe_i\otimes \bfe_j+(-1)^{|i|\cdot |j|} \bfe%
_j\otimes \bfe_i\right) \,\text{ and }\ y_{i,j}:= \bfe^*_j\otimes \bfe^*_i+
(-1)^{|i|\cdot|j|} \bfe^*_i\otimes\bfe^*_j,
\end{equation}
for $i,j\in\mathcal{I}_{m,n}$. The $x_{i,j}$ span $W$, and therefore they
generate $\sS(W)$. Similarly, the $y_{i,j}$ span $W^*$, and therefore they
generate $\sP(W)$. Moreover, 
\begin{equation*}
x_{i,j}=(-1)^{|i|\cdot |j|}x_{j,i},\ y_{i,j}=(-1)^{|i|\cdot |j|}y_{j,i},\ 
\text{and } \lag y_{i,j},x_{p,q}\rag=\delta_{i,p}\delta_{j,q}
+(-1)^{|i|\cdot |j|}\delta_{i,q}\delta_{j,p}. 
\end{equation*}

%In this fashion we can fix isomorphisms of 
%$\C$-superalgebras 
%\[
%\sS(W)\cong \C[x_{i,j}]_{i,j\in\mathcal I}.
%\text{ and }
%\sP(W)\cong \C[y_{i,j}]_{i,j\in\mathcal I}
%\] 
The action of $\g g$ on $\sS(W)$ can be realized %on
%$\C[x_{i,j}]_{i,j\in\mathcal I_{m,n}}$
by the polarization operators 
\begin{equation}  \label{polrhoEi-j}
\rho(E_{i,j}):=\sum_{r\in\mathcal{I}_{m,n}}x_{i,r}\sfD_{{j,r}}
\ \text{ for }i,j\in\mathcal I_{m,n},
\end{equation}
where $\sfD_{{i,j}}:\sS(W)\to \sS(W)$ is the superderivation of parity $%
|i|+|j|$ uniquely defined by 
\begin{equation}  \label{pari-j}
\sfD_{i,j}(x_{k,l}):=\delta_{i,k}\delta_{j,l}+(-1)^{|i|\cdot|j|}\delta_{i,l}%
\delta_{j,k} \text{ for }i,j,k,l\in\mathcal{I}_{m,n}.
\end{equation}
%Then 
%$
%\sP(W)\cong\sS(W^*)\cong \C[y_{i,j}]_{i,j\in\mathcal %I_{m,n}}^{}
%$, and
Similarly, the action of $\g g$ on $\sP(W)$ is realized by the polarization
operators 
\begin{equation}  \label{ppooll}
\check\rho(E_{i,j})=-(-1)^{|i|\cdot|j|} \sum_{r\in\mathcal{I}_{m,n}}
(-1)^{|r|}y_{r,j}\partial_{{r,i}}\ \text{ for }
i,j\in\mathcal I_{m,n},
\end{equation}
where $\partial_{{i,j}}:=\partial_{x_{i,j}}$ is the superderivation of $%
\sP(W)$ corresponding to $x_{i,j}\in W$, as in
\eqref{pww*}.  % we have $\partial_{r,i}=(-1)^{|r|+|i|}\partial_{x_{r,i}}$.

Let $U=U_\eev\oplus U_\ood$ be a $\Ztwo$-graded vector space. For every $\cA%
\subseteq\mathrm{End}_\C(U)$, we set 
\begin{equation*}
\cA^{\prime }:= \{B\in\mathrm{End}_\C(U)\ :\ [A,B]=0 \,\text{ for every }A\in%
\cA
\}. 
\end{equation*}
\begin{rmk}
\label{rmk-dblcomm}
Fix finite dimensional $\Ztwo$-graded vector spaces $U$ and $U'$, and set 
\[
\cA:=\mathrm{End}_\C(U)\otimes 1_{U'}\subset
\mathrm{End}_\C(U\otimes U').
\]
Then $\cA'=1_U\otimes \mathrm{End}_\C(U')$.
\end{rmk}
%The next two lemmas are probably well known. However, since we could not find an explicit reference, for the sake of completeness we provide the proofs. 
\begin{lem}
\label{UAcomm}
Let $U=U_\eev\oplus U_\ood$ be a finite dimensional $\Ztwo$-graded vector space, and let $\cA\subseteq\mathrm{End}_\C(U)_\eev$ 
be a semisimple associative algebra. Then $(\cA')'=\cA$.
\end{lem}
\begin{proof}
Since $\cA$ is semisimple and purely even, 
both $U_\eev$ and $U_\ood$ can be expressed as direct sums of irreducible $\cA$-modules. It follows that
%\begin{equation}
%\label{V2dUt}
$U\cong\bigoplus_{\tau} 
U_\tau\otimes V_\tau,
$ 
%\end{equation}
where the $U_\tau$ are mutually non-isomorphic 
irreducible $\cA$-modules and  
$V_\tau\cong
\Hom_{\cA}(U_\tau,U)$. Since
the decomposition of $U$ into irreducibles is homogeneous,  
the multiplicity spaces $V_\tau$ are 
$\Ztwo$-graded, whereas the $U_\tau$ are purely even. 
By Artin-Wedderburn theory \cite{Jacobson},
$\cA=\bigoplus_{\tau}
\mathrm{End}_\C(U_\tau)\otimes 1_{V_\tau}^{}$.
Thus by Remark \ref{rmk-dblcomm},
$
\cA'=\bigoplus_{\tau}1_{U_\tau}^{}\otimes \mathrm{End}_\C(V_\tau)
$ and
$
(\cA')'=
\bigoplus_{\tau}
\mathrm{End}_\C(U_\tau)\otimes 1_{V_\tau}^{}=\cA
$.
\end{proof}

\begin{lem}
\label{lem-supcmmt}
For every $d\geq 1$, the superalgebra $\mathrm{End}_\g g(V^{\otimes d})$ 
 is spanned by the operators 
 $T^\sigma_{V,d}$  defined in 
 \eqref{dfTsigm}.
In particular, 
$
\mathrm{End}_\g g(V^{\otimes d})
=
\mathrm{End}_\g g(V^{\otimes d})_\eev
=\Mor_{\g g\text{-\upshape mod}}(V^{\otimes d},
V^{\otimes d}
)
$.
\end{lem}
\begin{proof}
%For every homogeneous subspace 
%$\mathcal S\subset\mathrm{End}_\C(V^{\otimes 2d})$,
%set 
%\[
%\mathrm{Comm}(\mathcal S):=\{
%\,
%T\in\mathrm{End}_\C(V^{\otimes 2d})\ :\ 
%[S,T]=0\text{ for all }S\in\mathcal S
%\,\}.
%\] 
Set $\cA:=\mathrm{Span}_\C\{T_{V,d}^\sigma\ :\ \sigma\in S_d\}\subset\mathrm{End}_\C(V^{\otimes d})_\eev$. Since $\cA$ is a homomorphic image of the group algebra $\C[S_d]$, it is a semisimple associative algebra, and 
 Lemma \ref{UAcomm} implies that
$(\cA')'=\cA$.

Let $\cB$ be the associative subalgebra of 
$\mathrm{End}_\C(V^{\otimes d})$ generated by the image of $\bfU(\g g)$.
From \cite[Theorem 1]{sergeev} or \cite[Theorem 4.14]{BereleRegev} it follows that
$\cA'=\cB$. Therefore we obtain that $\mathrm{End}_\g g(V^{\otimes d})=\cB'=(\cA')'=\cA$.
 %Since $\cB$ is semisimple,  $V^{\otimes 2d}$ is a completely reducible $\cB$-module, so that \begin{equation} \label{V2dUt}V^{\otimes 2d}\cong\bigoplus_{\tau\in \widehat{S}_{2d}} U_\tau\otimes U'_\tau \end{equation}where the $U_\tau$'s are distinct irreducible $S_{2d}$-modules and $U'_\tau\cong \Hom_{S_{2d}}(U_\tau,V^{\otimes 2d})$. In fact there is a $\Ztwo$-grading on the right hand side of \eqref{V2dUt}, which is induced by the tensor product of the purely even space $U_\tau$ and the $\Ztwo$-graded space $U_\tau'$. Furthermore, the $\Ztwo$-gradings on both sides of \eqref{V2dUt} coincide.But $\cA=\mathrm{Comm}(\cB)$, so that \[\cA=\bigoplus_{\tau\in\widehat{S}_d} I\otimes \mathrm{End}_\C(U_\tau').\]It is now straightforward to check that $\cB=\mathrm{Comm}(\cA)$.
\end{proof}
The canonical isomorphism 
\begin{equation}  \label{V2dV*2dE}
V^{\otimes 2d}\otimes V^{*\otimes 2d}\cong \mathrm{End}_\C(V^{\otimes 2d})
\end{equation}
given in \eqref{U'U*} maps $v_1\otimes\cdots\otimes v_{2d}\otimes
v_1^*\otimes \cdots\otimes v_{2d}^*\in V^{\otimes 2d}\otimes V^{*\otimes 2d} 
$ to the linear map 
\begin{equation*}
V^{\otimes 2d} \to V^{\otimes 2d}\ ,\ w_1\otimes \cdots\otimes w_{2d}
\mapsto \lag v^*_{2d},w_1\rag
\cdots \lag v_1^*,w_{2d}\rag
v_1\otimes\cdots\otimes v_{2d} .
\end{equation*}
For every $\sigma\in S_{2d}$, let $\widetilde\bft_\sigma\in V^{\otimes
2d}\otimes V^{*\otimes 2d}$ be the element corresponding to $
T^{\sigma^{-1}}_{V,2d}$ via the isomorphism \eqref{V2dV*2dE}, where
$
T^{\sigma^{-1}}_{V,2d}
$ 
is defined as in \eqref{dfTsigm}. 
It is easy to
verify that 
\begin{equation*}
\widetilde\bft_\sigma=\sum_{i_1,\ldots,i_d\in\mathcal{I}_{m,n}} (-1)_{}^{ \seps%
\left(\sigma;\bfe_{i_1},\ldots,\bfe_{i_{2d}}^{}\right) } \bfe_{i_{\sigma(1)}} \otimes
\cdots \otimes \bfe_{i_{\sigma(2d)}} \otimes \bfe_{i_{2d}}^{*}\otimes
\cdots\otimes \bfe_{i_1}^{*}, 
\end{equation*}
where $
\seps\left(\sigma;\bfe_{i_1},\ldots,\bfe_{i_{2d}}^{}\right)
$ is defined in \eqref{sepssign}. Now 
set $\sigma_\circ:=(1,2)\cdots (2d-1,2d)\in S_{2d}$ and \[
H_{2d}:=\{\sigma\in
S_{2d}\,:\,\sigma\sigma_\circ=
\sigma_\circ\sigma\}.
\] 
Let $\psf_d^{}:V^{\otimes 2d}\to
V^{\otimes 2d}$ and 
$\psf_d^{*}:V^{*\otimes 2d}\to
V^{*\otimes 2d}$
be the canonical projections
onto $\sS^d(W)$ and $\sS^d(W^*)\cong\sP^d(W)$, so that
\begin{equation}
\label{dfpsfdd*}
\psf_d^{}=\frac{1}{2^d d!}\sum_{\sigma\in H_{2d}}T_{V,2d}^{\sigma}
\ \text{ and }\ 
\psf_d^*=\frac{1}{2^d d!}\sum_{\sigma\in H_{2d}}T_{V^*,2d}^{\sigma}.
\end{equation}
For every $\sigma\in S_{2d}$, 
let 
$\bft_\sigma\in \sP^d(W)\otimes \sS^d(W)$ be defined by
\begin{equation*}
\bft_\sigma:= 
\left(
\braid^{}_{V^{\otimes
2d},V^{*\otimes 2d}}
\circ
(\psf^{}_d\otimes \psf^{*}_d) \right) \big(\widetilde\bft_\sigma\big),
\end{equation*}
where 
$\braid^{}_{V^{\otimes
2d},V^{*\otimes 2d}}$ is the symmetry isomorphism, defined in 
\eqref{symisomdf}.
The action of the linear transformation $T_{V^{},2d}^{\sigma_1^{}}\otimes
T_{V^*,2d}^{\sigma_2^{}}$ on 
$V^{\otimes 2d}\otimes V^{*\otimes 2d}\cong\mathrm{End}_\C(V^{\otimes 2d})$ is given by 
\[
T\mapsto T_{V,2d}^{\sigma_1^{}}T
T_{V,2d}^{\pi\sigma_2^{-1}\pi^{-1}}
\text{ for every }T\in
\mathrm{End}_\C(V^{\otimes 2d}),\]
where $\pi\in S_{2d}$ is defined by 
$\pi(i)=2d+1-i$ for every $1\leq i\leq 2d$. Thus from
\eqref{dfpsfdd*} it follows that 
\begin{align}  
\label{tsigmeqqqqn}
\bft_\sigma&= \left(
\frac{1}{d!2^d}\right)^2 \sum_{\sigma^{\prime },\sigma^{\prime \prime
}\in H_{2d}} 
\bft_{\sigma^{\prime }\sigma\sigma^{\prime \prime }}
\\
&= \frac{1}{2^d}\sum_{i_1,\ldots,i_{2d}\in\mathcal{I}_{m,n}} \!\!\!\!\!\!(-1)^{|{i_1}|+\cdots+|{i_{2d}
}|
+
\seps
\left(
\sigma;\bfe_{i_1}^{},\ldots,\bfe_{i_{2d}}^{}\right)} y_{i_{2d-1},i_{2d}}^{}
\cdots y_{i_{1},i_{2}}^{} \otimes x_{i_{\sigma(1)},i_{\sigma(2)}}^{} \cdots
x_{i_{\sigma(2d-1)},i_{\sigma(2d)}}^{}.  \notag
\end{align}

\begin{lem}
\label{lem2..5}
Let $\sigma\in S_{2d}$, $\sigma'\in S_{2d'}$, and $\sigma''\in S_{2d''}$ be such that $d'+d''=d$ and
\[
\sigma(r):=\begin{cases}
\sigma'(r)&\text{ if }1\leq r\leq 2d',\\
\sigma''(r-2d')+2d'&\text{ if }2d'+1\leq r\leq 2d.
\end{cases}
\] 
Then 
$\bft_\sigma=
\bft_{\sigma'}\bft_{\sigma''}$ as 
elements of the  superalgebra $\sP(W)\otimes \sS(W)$.
\end{lem}
\begin{proof}
Follows immediately from the explicit summation formula for $\bft_\sigma$ given in \eqref{tsigmeqqqqn}.
\end{proof}

\begin{prp}
\label{taud0dr}
Let $\sigma\in S_{2d}$. Then the following statements hold.
\begin{itemize}
\item[\upshape (i)] 
$\bft_\sigma=\bft_{\sigma'\sigma\sigma''}$
for every $\sigma',\sigma''\in H_{2d}$.

\item[\upshape (ii)] There exist $\sigma',\sigma''\in H_{2d}$
and integers  
$0=d_0<d_1< \cdots< d_{r-1}<d_r=d
$ such that 
\begin{equation}
\label{eqTAUdecc}
\sigma'\sigma\sigma''=\tau_{d_0,d_1}^{}
\tau_{d_1,d_2}^{}\cdots \tau_{d_{r-1},d_r^{}}^{},
\end{equation}
where  $\tau_{a,b}$ for $a<b$ denotes the cycle
$
(2a+2,2a+4,\ldots,2b-2,2b)
$.

\end{itemize}

\end{prp}

\begin{proof}
(i) Follows from \eqref{tsigmeqqqqn}.

\noindent (ii) 
First we show that there exist integers $a_1,\ldots,a_d,b_1,\ldots,b_d$ and a permutation $\tau\in S_d$ such that 
\begin{equation*}
2s-1\leq a_s,b_s\leq 2s
\text{ and }
\sigma(a_s)=b_{\tau(s)}
\text{ for }
1\leq s\leq d.
\end{equation*}
To this end, 
we construct an undirected  bipartite multigraph 
$\mathcal G$, 
with vertex set \[
\mathcal V
:=\{\mathbf a_1,\ldots,\mathbf a_d\}\cup\{\mathbf b_1,\ldots,\mathbf b_d\},
\] and 
$m_{s,s'}:=\left|M_{s,s'}\right|$ edges between 
$\mathbf a_s$ and $\mathbf b_{s'}$, labelled by elements of the set $M_{s,s'}$, where
\[
M_{s,s'}:=\{\sigma(2s-1),\sigma(2s)\}\cap \{2s'-1,2s'\}\,
\text{ for all }1\leq s,s'\leq d
.
\] 
Elements of $M_{s,s'}$ correspond to equalities of the form $\sigma(a_s)=b_{s'}$, where
$2s-1\leq a_s\leq 2s$ and $2s'-1\leq b_{s'}\leq 2s'$.
Every vertex in $\mathcal G$ has degree two, and
therefore $\mathcal G$ is a union of disjoint cycles. It follows that $\mathcal G$  has a perfect matching,
i.e., 
 a set of $d$ edges  which do not have a vertex in common. 
Using this matching we define $\tau\in S_d$ so that the  edges of the matching are $(\mathbf{a}_s,\mathbf b_{\tau(s)})$. The label on each edge 
$(\mathbf{a}_s,\mathbf b_{\tau(s)})$ of the matching is the corresponding $b_{\tau(s)}$, and $a_s:=\sigma^{-1}\big(b_{\tau(s)}\big)$.

Now let $\sigma_1,\sigma_2\in H_{2d}$ be defined by
\[
\sigma_1(t):=\begin{cases}
2s-1& \text{ if }t=2\tau(s)-1\text{ for }1\leq s\leq d,\\
2s& \text{ if }t=2\tau(s)\text{ for }1\leq s\leq d.
\end{cases}
\ \ \text{ and }\ \
\sigma_2:=\prod_
{\substack{1\leq s\leq d\\[.3mm]
a_s-b_{\tau(s)}\not\in 2\Z}}\!\!\!
(2s-1,2s),\]
and 
set $\sigma_3:=\sigma_2\sigma_1\sigma$. It is easy to check that that $\sigma_3(a_s)=a_s$ for $1\leq s\leq d$.
Next set 
\[
\sigma_4:=\prod_{\substack{1\leq s\leq d\\ a_s=2s}}(2s-1,2s)\text{\ \,and\,\ }\sigma_5:=\sigma_4\sigma_3\sigma_4^{-1},
\]
so that $\sigma_5(2s-1)=2s-1$ for $1\leq s\leq d$. Now let $\tau\in S_d$ be defined by $\tau(s):=\frac{1}{2}\sigma_5(2s)$ for $1\leq s\leq d$. Choose $0=d_0<d_1<\cdots<d_r=d$ 
and $\tau_\circ\in S_d$ such that
\[
\tau_\circ^{}\tau\tau_\circ^{-1}=
\prod_{s=1}^d(d_{s-1}+1,\ldots,d_s).
\]
Finally, let $\sigma_\circ\in H_{2d}$ be defined by
\[
\sigma_\circ(2s):=
2\tau_\circ(s)
\,\text{ and }\,
\sigma_\circ(2s-1):=
2\tau_\circ(s)-1
\,\text{ for }\,1\leq s\leq d
.\]
It is straightforward to check that $\sigma_\circ^{}\sigma_5\sigma_\circ^{-1}
=\tau_{d_0,d_1}^{}\cdots \tau_{d_{r-1},d_r^{}}^{}$,
so that \eqref{eqTAUdecc} holds for $\sigma':=\sigma_\circ^{}\sigma_4^{}\sigma_2^{}\sigma_1^{}$ and $\sigma'':=\sigma_4^{-1}\sigma_\circ^{-1}$.
\end{proof}

\begin{lem}
\label{prp-g-inv-sPWW}
For every $d\geq 1$, 
\begin{equation}
\label{sSdttsPd}
\left(
\sP^d(W)\otimes \sS^d(W)\right)^\g g 
=\mathrm{Span}_\C\left\{\bft_\sigma\,:\,\sigma\in S_{2d}\right\}
.
\end{equation}
Furthermore,
$\left(\sP(W)\otimes \sS(W)\right)^\g g
=\bigoplus_{d\geq 0}\left(
\sP^d(W)\otimes \sS^d(W)\right)^\g g$.
\end{lem}
\begin{proof}
For every $k,l\geq 0$, the action of $\g g$ on 
$\sP(W)\otimes\sS(W)$ 
leaves $\sP^k(W)\otimes \sS^l(W)$ invariant. It follows that $\left(\sP(W)\otimes \sS(W)\right)^\g g=\bigoplus_{k,l\geq 0}\left(\sP^k(W)\otimes \sS^l(W)\right)^\g g$. By considering
the action of the centre of $\g g$ we obtain
that \begin{equation}
\label{sSksSlk=l}
\left(\sP^k(W)\otimes \sS^l(W)\right)^\g g=\{0\} \text{ unless }k=l.
\end{equation}  
Next we prove \eqref{sSdttsPd}. Recall the definition of
$\psf_d^{}$ and $\psf_d^*$ from \eqref{dfpsfdd*}.
It suffices to prove that the canonical projection
\begin{equation}
\label{eq-split}
V^{*\otimes 2d}\otimes V^{\otimes 2d}
\xrightarrow{\psf^*_d\otimes\psf_d^{}}
\sS^d(W^*)\otimes \sS^d(W)
\end{equation}
is surjective on $\g g$-invariants. This follows from
$
\sS^d(W^*)\otimes \sS^d(W)=
\left(
V^{*\otimes 2d}\otimes V^{\otimes 2d}
\right)^{H_{2d}\times H_{2d}}
$
and the fact that $\psf_d^*\otimes \psf_d^{}$ restricts to the identity map on $\sS^d(W^*)\otimes \sS^d(W)$.
%. Since the action of $H_{2d}$ on $V^{*\otimes 2d}\otimes V^{\otimes 2d}$ is completely reducible and also commutes with the action of $\g g$, the map \eqref{eq-split} has a $\g g$-equivariant splitting. Surjectivity of \eqref{eq-split}on $\g g$-invariants follows immediately.
\end{proof}

For every $i,j\in\mathcal{I}_{m,n}$, let $\varphi_{i,j} \in \sP(W)\otimes\sS
(W)$ be defined by 
\begin{equation*}
\varphi_{i,j}:=(-1)^{|i|\cdot|j|} \sum_{r\in\mathcal{I}%
_{m,n}}(-1)^{|r|}y_{r,j}\otimes x_{r,i}. 
\end{equation*}
Note that 
\begin{equation}  \label{sfmPHIEIJ}
\sfm(\varphi_{i,j})=-\check\rho(E_{i,j}) .
\end{equation}

\begin{lem}
\label{lem2..8}
Fix $d\geq 1$ and let $\sigma:=(2,4,\ldots,2d)\in S_{2d}$, that is, $\sigma(2r)=2r+2$ for $1\leq r\leq d-1$, $\sigma(2d)=2$, and $\sigma(2r-1)=2r-1$ for $1\leq r\leq d$. Then
\[
\bft_{\sigma}
=(-2)^d\,\widehat\sfs_d\big(
\check\rho\big(\mathrm{str}\big(\mathbf E^d\big)\big)
\big),
\]
where $\mathbf E$ is the matrix defined in Lemma \ref{gelfzhel}.
\end{lem}
\begin{proof}
For  $i_1,\ldots,i_{2d}\in\mathcal I_{m,n}$, we 
set 
\[
\seps'(i_1,\ldots,i_{2d}):={\sum_{s=1}^{2d}|i_s|+\sum_{s=3}^{2d}
|i_2|\cdot|i_s|+\sum_{s=1}^{d-1}|i_{2s+1}|\cdot|i_{2s+2}|}
.
\]
By a straightforward but tedious sign calculation, we obtain from \eqref{tsigmeqqqqn} that 
\begin{align*}
\bft_\sigma&=
\frac{1}{2^d}
\sum_{i_1,\ldots,i_{2d}\in\mathcal I_{m,n}}
(-1)^{\seps'(i_1,\ldots,i_{2d})}_{}y_{i_{2d-1},i_{2d}}\cdots y_{i_{1},i_{2}}
\otimes x_{i_{1},i_{\sigma(2)}}\cdots x_{i_{2d-1},i_{\sigma(2d)}}\\
&
=
\frac{1}{2^d}\sum_{t_1,\ldots,t_{d}\in\mathcal I_{m,n}}
\left(
(-1)_{}^{|t_1|+\sum_{s=1}^d|t_{s}|\cdot |t_{s+1}|}
\prod_{s=1}^d\varphi_{t_{s+1},t_s}
\right)
\end{align*}
where $t_s:=i_{2s}$ for $1\leq s\leq d$ and $t_{d+1}:=i_2$.
Using \eqref{whsD1Dk} and \eqref{sfmPHIEIJ} we can write
\begin{align*}
\widehat\sfs_d\big(
\check\rho\big(\mathrm{str}\big(\mathbf E^d\big)\big)
\big)&=
\sum_{t_1,\ldots,t_d\in\mathcal I_{m,n}}
\!\!\!\!\!\!(-1)^{|t_1|+\sum_{s=1}^d|t_s|\cdot |t_{s+1}|}
\;\widehat\sfs_d\left(
\check\rho(E_{t_2,t_1})\cdots\check\rho(E_{t_1,t_d})
\right)\\
&=
\sum_{t_1,\ldots,t_d\in\mathcal I_{m,n}}
\!\!\!\!\!\!(-1)^{|t_1|+\sum_{s=1}^d|t_s|\cdot |t_{s+1}|}
\;
\widehat\sfs_1\left(
\check\rho(E_{t_2,t_1})\right)\cdots
\widehat\sfs_1\left(
\check\rho(E_{t_1,t_d})
\right)\\
&=(-2)^d\bft_\sigma.
\qedhere\end{align*}
\end{proof}

Let $V:=\C^{m|n}$,  $W:=\sS^2(V)$,  $\g g:=\gl(V)=\gl(m|n)$, and $\check\rho$ be as above. Our first main result is the  following theorem.
 \begin{thm}
\label{prpgw}
\emph{(Abstract Capelli Theorem for $W:=\sS^2(V)$.)} 
For every $d\geq 0$, we have 
\[
\sPD^{ d}(W)^{\g g}=
\check\rho\left(\mathbf{Z}(\g g)\cap \bfU^d(\g g)\right).
\]

\end{thm}

\begin{proof}
The inclusion $\supseteq $ is obvious by definition. The proof of the inclusion $\subseteq$ is by induction on $d$. The case $d=0$ is obvious. Next assume that the statement holds for all $d'<d$.
Fix $D\in\sPD^{d}(W)^{\g g}$ such that $\mathrm{ord}(D)=d$. It suffices to find 
$z_D\in\bfZ(\g g)\cap \bfU^{d}(\g g)$ 
such that 
$
\widehat\sfs_{d}
\left(\check\rho(z_D)-D\right)=0
$.\\

\noindent\textbf{Step 1.} 
Choose $\varphi^D\in\sP(W)\otimes \sS(W)$ such that $\sfm(\varphi^D)=D$. We can write
\[
\varphi^D=\varphi^D_0+
\cdots+\varphi^D_{d}
\text{ where }\varphi^D_r\in\sP(W)\otimes \sS^r(W)
\text{ for }0\leq r\leq d,
\]
and Lemma \ref{prp-g-inv-sPWW} implies that 
$\varphi^D_r\in\left(\sP^r(W)\otimes \sS^r(W)\right)^{\g g}$ for $0\leq r\leq d$.
Furthermore, 
\[
\widehat\sfs_{d}(D)=
\widehat\sfs_{d}\left(\sfm
\left(\varphi^D\right)\right)
=
\widehat\sfs_{d}\left(\sfm
\left(\varphi_{d}^D\right)\right)=
\varphi_{d}^D.
\]
By Lemma \ref{prp-g-inv-sPWW}, $\varphi^D_{d}$ is a 
linear combination of the tensors  
$\bft_{\sigma}$ for $\sigma\in S_{2d}$. 
Thus to complete the proof, it suffices to find $z_{ \sigma}\in\bfZ(\g g)\cap \bfU^{d}(\g g)$ such that 
$\widehat\sfs_{d}\left(
\check\rho(z_{\sigma})\right)
=\bft_{\sigma}$.\\

\noindent\textbf{Step 2.} 
Fix $\sigma\in S_{2d}$ and let $0=d_0<d_1\cdots<d_{r-1}<d_r=d$  be as in 
Proposition
\ref{taud0dr}(ii).
%and 
%$\tau_{d_s,d_{s+1}}$, $0\leq s\leq r-1$,
Set $\sigma_s:=\left(2,\ldots,2(d_{s+1}-d_{s})\right)\in S_{2(d_{s+1}-d_{s})}^{}$ for $0\leq s\leq r-1$.
From Proposition \ref{taud0dr}(i)
and Lemma \ref{lem2..5} it follows that 
\[
\bft_\sigma=\bft_{\sigma_0^{}}\cdots
\bft_{\sigma_{r-1}^{}}.
\]
Now set 
$z_s:=(-\frac{1}{2})^{d_{s+1}-d_{s}}\,\mathrm{str}
\left(\mathbf E^{d_{s+1}-d_{s}}\right)$ for 
$0\leq s\leq r-1$, so that
$z_s\in\bfZ(\g g)$ by Lemma \ref{gelfzhel}. 
By Lemma \ref{lem2..8},
\begin{equation}
\label{ds+1-ds}
\widehat\sfs_{d_{s+1}-d_{s}}\left(
\check\rho(z_s)\right)=\bft_{\sigma_s}
\text{ for }0\leq s\leq r-1.
\end{equation}
From \eqref{ds+1-ds} it follows that
\begin{align*}
\widehat\sfs_{d}(\check\rho(z_0\cdots z_{r-1}))
&=
\widehat\sfs_{d}(\check\rho(z_0)\cdots \check\rho(z_{r-1}))\\
&=\widehat\sfs_{d_1-d_0}(\check\rho(z_0))
\cdots
\widehat\sfs_{d_{r}-d_{r-1}}(\check\rho(z_{r-1}))
=
\bft_{\sigma_0^{}}\cdots
\bft_{\sigma_{r-1}^{}}=\bft_{\sigma},
\end{align*}
which, as mentioned in Step 1, completes the proof of the theorem.
\end{proof}

%\begin{prp}
%\label{prpChWaSe}
%As a $\g g$-module,
%$
%\sP(W)=\bigoplus_{\lambda}V_\lambda
%$,
%where the direct sum is over all highest weights 
%$
%\lambda:=\sum_{i=1}^m\lambda_i\eps_i
%+
%\sum_{j=1}^n\lambda_{m+j}\eps_{\oline j}
%$
%for which there exists an even $(m|n)$-hook partition
%$\flat_\lambda=(\flat_1,\flat_2,\ldots)$
%such that 
%$\lambda_i=-\max\{\flat_{m+1-i}-n,0\}$
%for $1\leq i\leq m$, 
%and 
%$
%\lambda_{m+j}=-\flat'_{n+1-j}
%$
%for $1\leq j\leq n$.
%\end{prp}
%\begin{proof}
%Note that
%$\sP(W)
%=\bigoplus_{\mu}V_\mu^*$, where the direct sum is over %highest weights $\mu$ which appear in Proposition \ref{prpSWW}. 
%The contragredient representations $V_\mu^*$ are irreducible modules with highest weight $-\mu^-$, where $\mu^-$ is the highest weight of
%$V_\mu$ corresponding to the opposite Borel subalgebra.
%The calculation of  $\mu^-$, given in \cite[Thm 6.1]{OlPr} and \cite[Sec. 2.4.1]{ChWabook}, completes the proof of the proposition.
%\end{proof}

\section{The symmetric superpair $\big(\gl(m|2n),\g{osp}(m|2n)\big)$}

\label{Secsuperpr}

From now on, we assume that the odd part of $V$ is even dimensional. In other words, we set $V:=\C^{m|2n}$. We set $W:=\sS^2(V)$ and $\g g:=\gl
(V)\cong\gl(m|2n)$, as before. Let $\beta\in W^*$ be a nondegenerate, symmetric, even,
bilinear form on $V$, so that for every homogeneous $v,v'\in V$ we have $\beta(v,v^{\prime })=0 $ if $|v|\neq
|v^{\prime }|$, and 
$\beta(v^{\prime},v)=(-1)^{|v|\cdot |v^{\prime}|}\beta(v,v^{\prime }) $. Since $\beta\in W^*$ is even, we have
\begin{equation*}
(\check\rho(x)\beta)(w)=-(-1)^{|x|\cdot |\beta|}\beta(\rho(x)
w)=-\beta(\rho(x)w) \text{ for }x\in\g g,\ w\in W. 
\end{equation*}
Set $\g k:=\g{osp}(V,\beta):=\{x\in\gl(V)\, :\, \check\rho(x)\beta=0\}$, so
that 
\begin{equation*}
\g k=\left\{x\in\g g\, :\, \beta(x\cdot v,v^{\prime }) +(-1)^{|x|\cdot |v|}
\beta(v,x\cdot v^{\prime })=0\text{ for }v,v^{\prime }\in V \right\}. 
\end{equation*}
In this section we prove that every irreducible $\g g$-submodule of $\sP(W)$
contains a non-zero $\g k$-invariant vector.  
We remark that this statement does not follow from
\cite[Thm A]{AllSch}.
%\begin{rmk}
%The element $w_\lambda$ can be described more explicitly as follows. Let $v_1,\ldots,v_r\in V_\lambda$ be a basis for $V_\lambda$, and 
%$v_1^*,\ldots,v_r^*\in V_\lambda^*\subset \sS^d(W)$
%be the corresponding dual basis. We assume that the two bases consist of homogeneous elements. It is straightforward to verify that
%$
%w_\lambda=\sum_{i=1}^r v_i\otimes v_i^*
%$.
%\end{rmk}
%Our next goal is to construct a non-zero $\g k$-invariant vector $d_\lambda\in V_\lambda$. 

Recall the decomposition
$\g g=\g n^-\oplus\g h\oplus\g n^+$
from Section \ref{prfof571}.
Let $J_{2n}$ be the $2n\times 2n$ block-diagonal matrix defined by 
\begin{equation*}
J_{2n}:=\mathrm{diag}(\underbrace{J,\ldots,J}_{n\text{ times}}) \text{ where 
} J:=%
\begin{bmatrix}
0 & 1 \\ 
-1 & 0%
\end{bmatrix}
.
\end{equation*}
Let $I_m$ denote the $m\times m$ identity matrix. Without loss of generality
we can fix a homogeneous basis $\{\bfe_i\, :\, i\in\mathcal{I}_{m,2n}\}$ 
for $V\cong\C^{m|2n}$ such that 
\begin{equation}  \label{beiej=c}
\left[ \beta(\bfe_i{},\bfe_j^{}) \right]_{i,j\in\mathcal{I}_{m,2n}} =%
\begin{bmatrix}
I_m & 0 \\ 
0 & J_{2n}%
\end{bmatrix}%
.
\end{equation}

Set $\g a:=\left\{h\in\g h\,:\,\eps_{\oline{2k-1}}^{}(h)=\eps_{\oline{2k}%
}^{}(h)\text{ for }1\leq k\leq n\right\}$ and 
\begin{equation*}
\g t:=\left\{h\in\g h\,:\,\eps_k^{}(h)=0\text{ for }1\leq k\leq m\text{ and }%
\eps_{\oline{2l-1}}^{}(h)=-\eps_{\oline{2l}}^{}(h)\text{ for }1\leq l\leq
n\right\}. 
\end{equation*}
Then 
$\g t=\g k\cap \g h$ and 
$\g h=\g t\oplus\g a$. Set 
\begin{equation}
\label{gamkldf}
\gamma_k:=\eps_k\big|_{\g a} \text{ for }1\leq k\leq m\text{ and } \gamma_{%
\oline{l}}^{}:=\eps_{\oline{2l}}^{}\big|_{\g a} \text{ for }1\leq l\leq n. 
\end{equation}
The restricted root system of $\g g$ corresponding to $\g a$ will be denoted
by $\Sigma:=\Sigma^+\cup\Sigma^-$, where $\Sigma^+:=\Sigma_\eev^+\cup\Sigma_%
\ood^+$ is explicitly given by 
\begin{equation*}
\Sigma_\eev^+:=\Big\{\gamma_k^{}-\gamma_l^{}\Big\}
_{1\leq k<l\leq m}\cup \Big\{\gamma_{\oline k}-\gamma_{\oline l}\Big\}
_{1\leq k<l\leq n} \ \text{ and }\ \Sigma_\ood^+:=\Big\{\gamma_k^{}-\gamma_{%
\oline{l}}^{}\Big\}
_{1\leq k\leq m,\,1\leq l\leq n}. 
\end{equation*}
Furthermore, as usual $\Sigma^-=-\Sigma^+$.

For every $\gamma\in\Sigma
$ we set $\g g_{\gamma}:=\{x\in\g g\,:\,[h,x]=\gamma(h)x\text{ for every }%
h\in\g a\}$. Then 
\begin{equation}  \label{iwasawa}
\g g=\g k\oplus\g a \oplus\g u^+\text{ where } \g u^\pm:=\bigoplus_{\gamma%
\in\Sigma^\pm}\g g_\gamma\subseteq\g n^\pm.
\end{equation}

\begin{rmk}
\label{rmkbasisK}
Recall from Section \ref{prfof571} that
using the basis $\{\bfe_i\, :\, i\in\mathcal I_{m,2n}\}$ of $V=\C^{m|2n}$, 
we can represent every element of $\g g$ 
as an $(m,2n)$-block matrix whose rows and columns are indexed by the elements of $\mathcal I_{m,2n}$.
The following matrices form a spanning set of $\g k$.
\begin{itemize}
\item[(i)] $E_{k,l}-E_{l,k}$ for $1\leq k\neq l\leq m$,
\item[(ii)] $E_{\oline{2l-1},\oline{2k-1}}-E_{\oline{2k},\oline{2l}}$ for $1\leq k,l\leq n$,
\item[(iii)]
$E_{\oline{2l-1},\oline{2k}}+E_{\oline{2k-1},\oline{2l}}$ for $1\leq k,l\leq n$,
\item[(iv)] 
$E_{\oline{2l},\oline{2k-1}}+E_{\oline{2k},\oline{2l-1}}$ for $1\leq k,l\leq n$,
\item[(v)] $E_{k,\oline{2l-1}}+E_{\oline{2l},k}$
for $1\leq k\leq m$ and $1\leq l\leq n$,
\item[(vi)] $E_{k,\oline{2l}}-E_{\oline{2l-1},k}$
for $1\leq k\leq m$ and $1\leq l\leq n$.
\end{itemize}

\end{rmk}

Every irreducible finite dimensional representation of $\g g$ is a highest
weight module. Unless stated otherwise, the highest weights that we
consider will be with respect to the standard Borel subalgebra $\g b:=\g %
h\oplus\g n^+$. In this case, the highest weight $\lambda\in\g h^*$ can be written as
\begin{equation*}
\lambda=\sum_{k=1}^m\lambda_k^{}\eps_k^{}+ \sum_{l=1}^{2n}\lambda_{m+l}^{}\eps_{%
\oline l}\in\g h^*, 
\end{equation*}
such that
$\lambda_k-\lambda_{k+1}\in\Z_{\geq 0}$ for every 
$k\in\{1,\ldots,m-1\}\cup\{m+1,\ldots,m+2n-1\}$, where 
$\Z_{\geq 0}:=\{x\in\Z\,:\,x\geq 0\}$.
%The highest weight module corresponding to the highest weight $\lambda\in\g h^*$ will be denoted by $V_\lambda$.

\begin{lem}
\label{lem-dimleq1}
Let $U$ be an irreducible finite dimensional  $\g g$-module. 
Then $\dim U^\g k\leq  1$.
\end{lem}
\begin{proof} 
Let $\lambda\in\g h^*$ be the highest weight of $U$ and set $U':=\bigoplus_{\mu\neq\lambda}U(\mu)$, where $U(\mu)$ is the $\mu$-weight space of $U$. Clearly $U'$ is a $\Ztwo$-graded subspace of $U$ of codimension one, which is invariant under the action of $\g a\oplus\g u^-$. Note also that $U^\g k$ is a $\Ztwo$-graded subspace of $U$.
Suppose that 
$\dim U^\g k\geq  2$, so that $U'\cap U^\g k\neq\{0\}$. Fix 
a non-zero homogeneous vector
$u\in  U'\cap U^\g k$.
The PBW Theorem and the Iwasawa decomposition
\eqref {iwasawa} imply that
\[
U=\bfU(\g g)u=
\bfU(\g a\oplus \g u^-)\bfU(\g k)u
=\bfU(\g a\oplus\g u^-)u\subseteq U',
\]
which is a contradiction.
\end{proof}

A \emph{partition} is a sequence $\flat:=(\flat_1,\flat_2,\flat_3,\ldots)$ 
 of
integers satisfying $
\flat_k\geq \flat_{k+1}\geq 0
$ for every $k\geq 1$, and
$\flat_k=0$ for all but finitely many 
$k\geq 1$. The \emph{size} of $\flat$
is defined to be $|\flat|:=\sum_{k=1}^\infty \flat_k$. The \emph{transpose} of $
\flat$ is denoted by $\flat^{\prime }:=(\flat_1^{\prime },\flat_2^{\prime
},\flat_3^{\prime }\ldots)$, where \[
\flat_k^{\prime }:=|\{l\geq 1\ :\
\flat_l\geq k\}|.
\] 
\begin{dfn}
\label{dfnhookprn}
A partition 
$(\flat_1,\flat_2,\flat_3,\ldots)$ that satisfies
$\flat_{m+1}\leq n$
is called an \emph{$(m|n)$-hook partition}.
%An $(m|n)$-hook partition is called \emph{even}, if 
%$\flat_k$ is an even integer for  every $k\geq 1$. 
The set of $(m|n)$-hook partitions of size $d$ will be denoted by $\mathrm{H}_{m,n,d}$, and we set
$\mathrm{H}_{m,n}:=\bigcup_{d=0}^\infty
\mathrm{H}_{m,n,d}$.
\end{dfn}

%Recall that $\beta\in W^*$ is a nondegenerate, symmetric, even bilinear form on $V$ (see Section \ref{Secsuperpr}). 

%Next let
%$\pi:\sP(W)\otimes \sS(W)\to \sP(W)$ be defined by
%\[
%\pi(w\otimes w'):=\widetilde\beta(w')w
%\ \text{ for }\
%w\otimes w'\in 
%\sP(W)\otimes \sS(W).
%\]
%Note that $\pi=1\otimes \widetilde \beta$, so that 
%$\pi$ is also a homomorphism of Lie superalgebras.
%Furthermore, 
%\[
%\pi\left(\sP^d(W)\otimes \sS(W)\right)\subseteq 
%\sP^d(W)
%.
%\]

We now define a map \begin{equation}
\label{bfGam}
\boldsymbol{\Gamma}:\mathrm{H}_{m,n}\to \g a^*
\end{equation} 
as
follows.  To every $\flat=(\flat_1,\flat_2,\flat_3,\ldots)\in\mathrm{H}%
_{m,n}$ we associate the element $\lambda=\boldsymbol{\Gamma}(\flat)\in\g %
a^*$ given by 
\begin{equation}
\label{frmLamBk}
\lambda:=\sum_{k=1}^m2\flat_k\gamma_k+ \sum_{l=1}^n2\max\big\{\flat^{\prime
}_l-m,0\big\}\gamma_{\oline l} \in \g a^*. 
\end{equation}
For integers $m,n,d\geq 0$, set 
\begin{equation*}
\mathrm{E}_{m,n,d}:=\boldsymbol{\Gamma}(\mathrm{H}_{m,n,d}):= \big\{
\lambda\in\g a^*\,:\,\lambda=\boldsymbol{\Gamma}(\flat)\text{ for some }%
\flat\in \mathrm{H}_{m,n,d} \big\}
\end{equation*}
and
\begin{equation}
\label{Emndf}
\mathrm{E}_{m,n}:= \bigcup_{d=0}^\infty\mathrm{E}_{m,n,d}.
\end{equation}

Because of the decomposition $\g h=\g a\oplus\g t$, every $\lambda\in\g a^*$
can be extended trivially on $\g t$ to yield an element of $\g h^*$. It is
straightforward to verify that for every $\lambda\in\mathrm{E}_{m,n}$, the
above extension of $\lambda$ to $\g h^*$ is the highest weight of an irreducible $\g g
$-module. We will denote the latter highest weight module by $V_\lambda$.
From \cite[Thm 3.4]{ChWa} or \cite{brini} it follows that  
\begin{equation}  \label{prpSWW}
\sS^d(W)\cong\bigoplus_{\lambda\in\mathrm{E}_{m,n,d}} \!V_{\lambda}\ \ \ 
\text{as \,$\g g$-modules}.
\end{equation}
\begin{rmk}
\label{prpChWaSe}
From \eqref{prpSWW} it follows that 
\begin{equation}
\label{prpSWW*}
\sP^d(W)\cong\bigoplus_{\lambda\in\mathrm{E}_{m,n,d}}
V_\lambda^*
\cong
\bigoplus_{\mu\in\mathrm{E}^*_{m,n,d}}
V_\mu^{}
\end{equation}
where
$\mathrm{E}_{m,n,d}^*$ is the set of all $\mu\in\g a^*$ 
of the form
\begin{equation}
\label{eqhwofVl}
\mu=-\sum_{i=1}^m 2\max\{\flat_{m+1-i}-n,0\}\gamma_i-\sum_{j=1}^n 2\flat'_{n+1-j}\gamma_{\oline j} 
\text{ for some }\flat\in\mathrm{H}_{m,n,d},
\end{equation}
and $V_\mu$ denotes the $\g g$-module whose highest weight is the extension of $\mu$ trivially on $\g t$ 
to $\g h^*$.
 The decomposition \eqref{prpSWW*} follows from the fact that
for every $\lambda\in \mathrm{E}_{m,n,d}$, the contragredient module $V_\lambda^*$ has highest weight equal to the extension to $\g h^*$ of $-\lambda^-$, where $\lambda^-$ is the highest weight of
$V_\lambda$ with respect to the opposite Borel subalgebra $\g b^-:=\g h\oplus\g n^-$.
From the calculation of  $\lambda^-$ given in 
\cite[Thm 6.1]{OlPr} or \cite[Sec. 2.4.1]{ChWabook},
it follows that  the highest weight of $V_\lambda^*$ is of the form \eqref{eqhwofVl}.
\end{rmk}
Set 
\begin{equation}
\label{Emn*df}
\mathrm{E}_{m,n}^*:=\bigcup_{d=0}^\infty \mathrm{E}^*_{m,n,d}.
\end{equation} By the $
\g g$-equivariant isomorphism \eqref{eq-sfm:}, 
\begin{align}  \label{spdgCB}
\displaystyle\sPD(W)^\g g_{} &\cong \big(\sP(W)\otimes\sS(W)\big)^\g g \cong
\bigoplus_{\lambda,\mu\in\mathrm{E}_{m,n}^{*}} (V_\lambda\otimes V_\mu^*)^\g %
g \cong \bigoplus_{\lambda,\mu\in\mathrm{E}_{m,n}^{*}} \Hom_\g %
g(V_\mu,V_\lambda).
\end{align}
Every non-zero element of $\Hom_\g g(V_\mu,V_\lambda)$ should map $V_\mu^{\g %
n^+}$ to $V_\lambda^{\g n^+}$, and is uniquely determined by the image of a
highest weight vector in $V_\mu^{\g n^+}$. But $\dim(V_\lambda^{\g %
n^+})=\dim(V_\mu^{\g n^+})=1$, so that $\dim\left(\Hom_\g g(V_\lambda,V_\mu)\right)\leq 1 $
with equality if and only if $\lambda=\mu$.

\begin{dfn}
\label{DefDla}
For  $\lambda\in\mathrm{E}_{m,n}^*$, 
let $D_\lambda\in \sPD(W)^\g g$
be the $\g g$-invariant differential operator that corresponds 
via
the sequence of isomorphisms \eqref{spdgCB}
to  $1_{V_\lambda}^{}\in\mathrm{End}_\C(V_\lambda)$.
\end{dfn}
From \eqref{spdgCB} it follows that $\{D_\lambda\,:\,\lambda\in \mathrm{E}^*
_{m,n}\}$ is a basis for $\sPD(W)^\g g$. In the purely even case (that is,
when $n=0$), the latter basis is sometimes called the \emph{Capelli basis}
of $\sPD(W)^\g g$.

Let $\beta\in W^*$ be the symmetric bilinear form defining $\g k$, as at the beginning of Section \ref{Secsuperpr}. 
Let 
$\sfh_\beta:\sS(W)\to \C$ be the extension of $\beta$, defined as in \eqref{dfheta}. Set 
\begin{equation}
\label{tildhbeta}
\widetilde\sfh_\beta:\sP(W)\otimes \sS(W)\to 
\sP(W) \ ,\ a\otimes b\mapsto \sfh_\beta(b)a.
\end{equation}
Recall the map $\sfm:\sP(W)\otimes\sS(W)\to \sPD(W)$ defined in \eqref{eq-sfm:}.
\begin{prp}
\label{le-dlam}
Let  $\lambda\in\mathrm{E}^*_{m,n}$.  
Then $\dim V_{\lambda}^\g k=1$, and the vector
\[
\mathbf d_\lambda:=(\widetilde \sfh_\beta\circ \sfm^{-1})(D_\lambda).
\]
is a   
non-zero $\g k$-invariant vector of $V_\lambda$.
\end{prp}
\begin{proof}
Since  $\beta\in (W^*)^\g k$, the map $\sfh_\beta$ 
is 
$\g k$-equivariant.
It follows that 
$\widetilde\sfh_\beta=\yekP\otimes \sfh_\beta$ is
 $\g k$-equivariant, and therefore 
 $\mathbf d_\lambda\in V_\lambda^\g k$.
By Lemma \ref{lem-dimleq1}, it suffices to 
 prove that $\mathbf d_\lambda\neq 0$.
Fix $d\geq 0$ such that $\lambda\in \mathrm{E}_{m,n,d}^*$, so that $V_\lambda\in \sP^d(W)$ and $V_\lambda^*\subseteq \sS^d(W)$.\\

%For every $w_1,\ldots,w_d\in W$, set 
%\[w_1\bullet\cdots \bullet w_d:=\Sigma_d(w_1\otimes\cdots\otimes w_d).\]
%The action of $\g g$ on $\sS^d(W)$ is given by
%\[X\cdot(w_1\bullet\cdots \bullet w_d)=\sum_{i=1}^d(-1)^{|X|(|w_1|+\cdots+|w_{i-1}|)}w_1\otimes \cdots \otimes Xw_i\otimes \cdots\otimes w_d.\]
%In particular, 
%$\widetilde\beta_d(X\cdot(w_1\bullet\cdots \bullet w_d))=0$ for every $X\in\g k$. It follows that
%\begin{align*}
%\pi_d(X\cdot (w\otimes w'))&=
%\pi_d(X\cdot w\otimes w'+(-1)^{|X|\cdot|w|}w\otimes X\cdot w')\\
%&=\widetilde\beta_d(w')X\cdot w
%+
%(-1)^{|X|\cdot|w|}
%\widetilde\beta_d(X\cdot w')w
%=X\cdot\left(
%\widetilde\beta_d(w') w
%\right).
%\end{align*}
%Consequently, $\pi_d(X\cdot w)=X\cdot \pi_d(w)$ for $w\in \sS^d(W)$ and $X\in\g k$. Setting $w:=w_\lambda$, it follows that $d_\lambda$ is $\g k$-invariant.\\

\noindent\textbf{Step 1.} 
The linear map
$\sfh_\beta\big|_{V_\lambda^*}:V_\lambda^*\to\C$ can be represented by evaluation at a vector $v_\circ\in V_\lambda$, that is, $\sfh_\beta(v^*):=\lag v^*,v_\circ\rag$ for every $v^*\in V_\lambda^*$.

Let $\iota_\lambda:V_\lambda\otimes V_\lambda^*\to \Hom_\C(V_\lambda,V_\lambda)$ 
be the
isomorphism \eqref{U'U*}. 
For every 
$v\otimes v^*\in V_\lambda\otimes V_\lambda^*$,
\[
\widetilde\sfh_\beta
\circ\iota_\lambda^{-1}(T_{v\otimes v^*}^{})
=
\widetilde\sfh_\beta
(v\otimes v^*)
=
\sfh_\beta(v^*)v
=
\lag v^*,v_\circ\rag v
=
T_{v\otimes v^*}(v_\circ).
\]
It follows that  $\widetilde\sfh_\beta\circ \iota_\lambda^{-1}(T)=Tv_\circ$ for every $T\in\mathrm{End}_\C(V_\lambda)$. In particular, \[
\mathbf d_\lambda=\widetilde\sfh_\beta\circ\iota_\lambda^{-1}(1_\lambda)=v_\circ,
\] so that to complete the proof we need to show that $v_\circ\neq 0$, or equivalently, that $\sfh_\beta\big|_{V_\lambda^*}\neq 0$.\\

%\begin{equation}\label{beiej=c}\beta(\bfe_i,\bfe_j)=\begin{cases}1& \text{ if }1\leq i=j\leq m,\\1& \text{ if }i=\oline{2k-1}\text{ and }j=\oline{2k}\text{ where }1\leq k\leq n,\\ -1& \text{ if }j=\oline{2k-1}\text{ and }i=\oline{2k}\text{ where }1\leq k\leq n,\\ 0&\text{ otherwise.}\end{cases}\end{equation}

\noindent\textbf{Step 2.}
Let $\{x_{i,j}\ :\ {i,j\in\mathcal I_{m,2n}}\}$ be the
generating set  of 
$\sS(W)$ that is defined in 
\eqref{eq-dfxijyij}. We define the superderivations 
$\sfD_{i,j}$, for $i,j\in\mathcal I_{m,2n}$, as in \eqref{pari-j}.
Assume that 
$\sfh_\beta\big|_{V_\lambda^*}=0$. We  will reach a contradiction in Step 3.
 Fix $a\in V_\lambda^*\subseteq \sS^d(W)$. We claim that
\begin{equation}
\label{j1i1jkik}
\sfh_\beta\left(
\sfD_{j_1,i_1}\cdots\sfD_{j_k,i_k}(a)
\right)=0
\text{ for every }
k\geq 0,\ i_1,j_1,\ldots,i_k,j_k\in\mathcal I_{m,2n},
\end{equation}
where $\sfD_{i,j}$ is defined as in \eqref{pari-j}.
The proof of \eqref{j1i1jkik} is by induction on $k$. For
$k=0$, it follows from the assumption that $\sfh_\beta\big|_{V_\lambda^*}=0$. Next assume $k=1$. 
Set
\begin{equation}
\label{eqdfiprm}
i'_1:=\begin{cases}
i_1&\text{ if }|i_1|=\eev,\\
\oline{2k}&\text{ if }i_1=\oline{2k-1}\text{ where }
1\leq k\leq n,\\
\oline{2k-1}&\text{ if }i_1=\oline{2k}\text{ where }
1\leq k\leq n,
\end{cases}
\end{equation}
Then
$\rho\big(E_{i_1',j_1^{}}\big)a\in V_\lambda$, so that
$\sfh_\beta\big(\rho\big(E_{i_1',j_1^{}}\big)a\big)=0$. But
from \eqref{polrhoEi-j} it follows that
\begin{align*}
\sfh_\beta
\big(\rho\big(E_{i_1',j_1^{}}\big)a\big)
&=\sum_{r\in\mathcal I_{m,2n}}\sfh_\beta
\big(x_{i_1',r}\big)
\sfh_\beta\big(\sfD_{j_1,r}(a)\big)\\
&=
\sum_{r\in\mathcal I_{m,2n}}
\beta(\bfe_{i_1'},\bfe_r)
\sfh_\beta\big(\sfD_{j_1,r}(a)\big)=
\pm\sfh_\beta\big(\sfD_{j_1,i_1}(a)\big),
\end{align*}
so that $\sfh_\beta(\sfD_{j_1,i_1}(a))=0$.

Finally, assume that $k>1$. We define $i_1',\ldots,i_k'\in\mathcal I_{m,2n}$ from $i_1^{},\ldots,i_k^{}$ according to
\eqref{eqdfiprm}.
Then
$\rho\big(E_{i_1',j_1^{}}\big)\cdots \rho\big(E_{i_k',j_k^{}}\big)a\in V_\lambda$,
so that
$
\sfh_\beta\big(
\rho\big(E_{i_1',j_1^{}}\big)\cdots \rho\big(E_{i_k',j_k^{}}\big)a
\big)
=0$.
However,
% after moving all of the superderivations in $\rho\big(E_{i_1',j_1^{}}\big)\cdots \rho\big(E_{i_k',j_k^{}}\big)$ to the right, we obtain
we can write
\begin{equation}
\label{rhEkkk}
\rho\big(E_{i_1',j_1^{}}\big)\cdots \rho\big(E_{i_k',j_k^{}}\big)a
=
\sum_{r_1,\ldots,r_k\in\mathcal I_{m,2n}}
x_{i_1',r_1^{}}\cdots x_{i_k',r_k^{}}
\sfD_{j_1^{},r_1^{}}\cdots
\sfD_{j_k^{},r_k^{}}(a)
+R_a,
\end{equation}
where $R_a$ is a sum of elements of 
$\sS(W)$ of the form  
$b\sfD_{s_1,t_1}\cdots\sfD_{s_\ell,t_\ell}(a)$, with $b\in\sS(W)$ and $\ell<k$.
Since $\sfh_\beta$ is an algebra homomorphism, the induction hypothesis implies that $\sfh_\beta(R_a)=0$. 
Therefore \eqref{rhEkkk} implies that 
\begin{align*}
\sfh_\beta
\big(
\rho\big(E_{i_1',j_1^{}}\big)\cdots \rho\big(E_{i_k',j_k^{}}\big)a
\big)
&=
\sum_{r_1,\ldots,r_k\in\mathcal I_{m,2n}}
\sfh_\beta\big(x_{i_1',r_1^{}}\big)\cdots \sfh_\beta(x_{i_k',r_k^{}}\big)
\sfh_\beta\big(\sfD_{j_1^{},r_1^{}}\cdots
\sfD_{j_k^{},r_k^{}}(a)\big)\\
&=\pm \sfh_\beta
\big(\sfD_{j_1^{},i_1^{}}\cdots
\sfD_{j_k^{},i_k^{}}(a)\big),
\end{align*}
so that $\sfh_\beta
\big(\sfD_{j_1^{},i_1^{}}\cdots
\sfD_{j_k^{},i_k^{}}(a)\big)=0$.\\

\noindent\textbf{Step 3.}
Set $\mathcal J:=\{(k,\oline l)\,:\,1\leq k\leq m\text{ and }1\leq l\leq 2n\}\subset
\mathcal I_{m,2n}\times \mathcal I_{m,2n}$. 
Every $a\in\sS(W)$ can be written as
\[
a=\sum_{S\subseteq \mathcal J}a_Sx_S
\]
where $a_S\in\sS(W_\eev)$ and $x_S:=\prod_{(i,j)\in S}x_{i,j}$ for each $S\subseteq \mathcal J$. 
Now fix $S\subseteq\mathcal J$ and set $ \sfD_S:=\prod_{(i,j)\in S}\sfD_{i,j}$, so that $\sfD_S(a)=\pm a_S$.
From \eqref{j1i1jkik} it follows that 
\begin{equation}
\label{Hbetya}
\sfh_\beta\big(
\sfD_{j_1,i_1}
\cdots\sfD_{j_k,i_k})(a_S)\big)
=0\text{ for every }k\geq0,\
i_1,\ldots,i_k,j_1,\ldots,j_k\in\{1,\ldots,m\}.
\end{equation}
But $\sS(W_\eev)\cong \sP(W_\eev^*)$, and 
therefore \eqref{Hbetya} means that the multi-variable polynomial $a_S\in\sP(W_\eev^*)$ and all of its partial derivatives vanish at a fixed point $\beta_\eev:=\beta\big|_{W_\eev}\in W_\eev^*$. Consequently,  $a_S=0$. As $S\subseteq \mathcal J$ is arbitrary, it follows that
$a=0$.
\end{proof}

\begin{rmk}
\label{rmk-v*}
A slight modification of the proof of Proposition \ref{le-dlam}
shows that for every $\lambda\in\mathrm{E}_{m,n}$ we have
$\dim V_\lambda^\g k=1$ as well. To this end, instead of $\beta\in W^*$  we use the 
$\g k$-invariant vector 
\[
\beta^*:=-\frac{1}{4}(x_{1,1}+\cdots+x_{m,m})
+\frac{1}{2}(x_{\oline 1,\oline 2}+\cdots+x_{\oline{2n-1},\oline{2n}})
\in W.
\]
\end{rmk}

%\begin{lem}
%\label{lemlt=0}
%Let $\lambda\in\g h^*$ be such that 
%$V_\lambda\subset \sP(W)$. Then $\lambda\big|_\g t=0$.
%\end{lem}
%\begin{proof}
%The proof is similar to the argument for Lie algebras, but there is a small subtlety due to the nonexistence of compact real forms for Lie superalgebras. Fix a compact real form $\g k_{\eev,\R}\subset \g k_\eev$ such that $\g t\cap \g k_{\eev,\R}$ is a real form of $\g t$, and let $K$ denote the corresponding simply connected compact Lie group. 
%By integration with respect to the Haar measure of $K$ we can find 
%a $K$-invariant inner product 
%$( \cdot,\cdot)_\lambda^{}$ on
%$V_\lambda$. Now let $v_\lambda$ be a highest weight vector of $V_\lambda$ and $v_\circ $ be a nonzero $\g k$-fixed vector. The Iwasawa decomposition implies that $V_\lambda=\bfU(\g k)v_\lambda$, and 
%from the $\g k_{\eev,\R}$-invariance of
%$(\cdot,\cdot)_\lambda^{}$ it follows that 
%$(v_\circ,v_\lambda)_\lambda^{}\neq 0$.
%For every $X\in\g t\cap\g k_{\eev,\R}$, 
%\[
%\lambda(X)(v_\lambda,v_\circ)_\lambda^{}
%=
%( X\cdot v_\lambda,v_\circ)_\lambda^{}=
%-( v_\lambda,X\cdot v_\circ)_\lambda^{}
%=0,
%\]
%so that $\lambda(X)=0$. Since 
%$\g t\cap \g k_{\eev,\R}$ is a real form of $\g t$
%and $\lambda$ is $\C$-linear,  $\lambda\big|_{\g t}=0$.
%\end{proof}

\section{The eigenvalue and spherical polynomials $c_\protect\lambda$
and $d_\lambda$
}
\label{Sec-Sec5}
As in the previous section, we set $V:=\C^{m|2n}$, so that 
$\g g:=\gl(V)=\gl(m|2n)$.
Recall the decompostion of $\sP(W)$ into irreducible $\g g$-modules given in \eqref{prpSWW*} and \eqref{Emn*df}, that is,
\[
\sP(W)=\bigoplus_{\lambda\in\mathrm{E}_{m,n}^*}V_\lambda.
\]
For $\lambda\in \mathrm{E}_{m,n}^*$,
let  $D_\lambda\in\sPD(W)^\g g$ be the $\g g$-invariant differential operator as in
Definition \ref{DefDla}. 
%defined as the invariant differential operator corresponding to 
%$\yek_\lambda\in\Hom_\C(V_\lambda,V_\lambda)$ via
%\eqref{brepscycl}. 
Since the decomposition \eqref{prpSWW*} is multiplicity-free, for every $%
\mu\in\mathrm{E}_{m,n}^*$ the map 
\[
D_\lambda:V_\mu\to V_\mu
\] 
is
multiplication by  a scalar ${c}_\lambda(\mu)\in \C$. 
%The main theorem of this section (Theorem \ref{MAINTHM}) describes the precise relation between ${c}%_\lambda(\cdot)$ and the $\g k$-invariant vector $\mathbf{d}_\lambda$, defined in Proposition \ref{le-dlam}.

Fix $d\geq 0$ such that $\lambda\in\mathrm{E}_{m,n,d}^*$, so that $V_\lambda\sseq \sP^d(W)$. By Theorem \ref{prpgw} we
can choose $z_\lambda\in \mathbf{Z}(\g g)\cap\bfU^d(\g g)$ such that 
\begin{equation}
\label{choicezl}
\check\rho(z_\lambda)=D_\lambda.
\end{equation} Next choose $z_{\lambda,\g k}\in\g k\bfU(%
\g g)\cap\bfU^d(\g g)$, $z_{\lambda,\g u^+}\in\bfU(\g g)\g u^+\cap\bfU^d(\g %
g)$, and $z_{\lambda,\g a}\in\bfU(\g a)\cap\bfU^d(\g g)$ such that 
\begin{align}  \label{bfUgan}
z_\lambda=z_{\lambda,\g k}+z_{\lambda,\g u^+}+z_{\lambda,\g a},
\end{align}
according to the decomposition $\bfU(\g g)=\left( \g k\bfU(\g g) + \bfU(\g g)%
\g u^+\right)\oplus\bfU(\g a) $. For $\mu\in\g a^*$ let 
\[
\sfh_\mu:\bfU(%
\g a)\cong \sS(\g a)\to \C
\] be the extension of $\mu$ defined as in \eqref{dfheta}. 
\begin{lem}
\label{cmulZa}
$c_\lambda(\mu)=\sfh_\mu(z_{\lambda,\g a})$
for every $\lambda,\mu\in\mathrm{E}_{m,n}^*$.  In particular, $c_\lambda\in\sP(\g a^*)$.
\end{lem}
\begin{proof}
Let $v_\mu$ denote a highest weight vector of 
$V_\mu\subseteq\sP(W)$.
By Remark \ref{rmk-v*}, the contragredient module $V_\mu^*\subseteq\sS(W)$ contains a non-zero $\g k$-invariant vector 
$v_\circ^*$.  From the decomposition $\g g=\g k\oplus\g a\oplus\g u^+$ and the PBW Theorem it follows that $V_\mu=\bfU(\g k)v_\mu$. This means that if $\lag v_\circ^*,v_\mu\rag=0$, then
$\lag v_\circ^*,V_\mu\rag=0$, which is a contradiction. Thus $\lag v_\circ^*,v_\mu\rag \neq 0$, and
\begin{align*}
c_\lambda(\mu)\lag v_\circ^*,v_\mu\rag
 &
=\lag v_\circ^*,D_\lambda v_\mu\rag  
 =
 \lag v_\circ^*,\check\rho(z_\lambda)v_\mu\rag\\
& =
 \lag v_\circ^*, 
 (\check\rho(z_{\lambda,\g k})+
 \check\rho(z_{\lambda,\g u^+})+
 \check\rho(z_{\lambda,\g a}) 
 )v_\mu
 \rag
 =\lag
 v_\circ^*,\check\rho(z_{\lambda,\g a})v_\mu
\rag=\sfh_\mu(z_{\lambda,\g a})\lag v_\circ^*,v_\mu\rag.
\end{align*}
%so that $\boldsymbol c_\lambda(\mu)=\sfh_\mu(z_{\lambda,\g a})$.
It follows that $c_\lambda(\mu)=\sfh_\mu(z_{\lambda,\g a})$.
\end{proof}
\begin{dfn}
\label{dfcl}
For every $\lambda\in\mathrm{E}_{m,n}^*$, the polynomial
$c_\lambda\in\sP(\g a^*)$ is called
the \emph{eigenvalue polynomial} 
associated to $\lambda$.
\end{dfn}
\begin{rmk}
\label{rmkdgclam}
Note that by Theorem
\ref{prpgw},
if $\lambda\in\mathrm{E}_{m,n,d}^*$ then $\deg(c_\lambda)\leq d$.
\end{rmk}

\begin{lem}
\label{lem-clmuvanish}
Let $\lambda\in\mathrm{E}_{m,n,d}^*$. Then 
$c_\lambda(\lambda)=d!$
and
 $ c_\lambda(\mu)=0$ for all other
$\mu$ in $\bigcup_{d'=0}^d \mathrm{E}_{m,n,d'}^*$.
\end{lem}

\begin{proof}
Note that $D_\lambda\in\sfm(V_\lambda\otimes V_\lambda^*)\subseteq\sfm\left(\sP^d(W)\otimes\sS^d(W)\right)$
where $\sfm$ is the map defined in \eqref{eq-sfm:}. Thus if $d'<d$ then   $D_\lambda\sP^{d'}(W)=\{0\}$.
Next assume $d'=d$. Then
 the map
\begin{equation}
\label{sPDPWPW-}
\sPD(W)\otimes \sP(W)\to\sP(W)
\ ,\ D\otimes p\mapsto Dp 
\end{equation}
is $\g g$-equivariant, and since
$D_\lambda\in\sfm(V_\lambda\otimes V_\lambda^*)$, the map \eqref{sPDPWPW-} restricts to a $\g g$-equivariant map $V_\mu\to V_\lambda$ given by $p\mapsto D_\lambda p$. By Schur's Lemma, when $\lambda\neq \mu$, the latter map should be zero. 

Finally, to prove that 
$\mathbf c_\lambda(\lambda)=d!$, we consider the bilinear form
\begin{equation}
\label{bilformbb}
\boldsymbol\beta:\sP^d(W)\times \sS^d(W)\to\C\
,\  
\boldsymbol\beta(a,b):=\partial_ba.
\end{equation}
By a straightforward calculation one can verify that   $\boldsymbol\beta(a,b)=d!\lag a,b\rag$, where $\lag\cdot,\cdot\rag$ denotes the duality pairing 
between $\sP^d(W)\cong\sS^d(W)^*$ and $\sS^d(W)$.
Next we choose a basis $v_1,\ldots,v_t$  
for $V_\lambda$. Let $v_1^*,\ldots,v_t^*$ be the corresponding dual basis of $V_\lambda^*$. 
From \eqref{bilformbb} it follows that
\[
D_\lambda v_k=\sfm(\sum_{l=1}^t v_l\otimes v_l^*)v_k=\sum_{i=1}^tv_l\partial_{v_l^*}(v_k)=
d! v_k
\] for every $1\leq k\leq t$.
\end{proof}

Let $\beta^*\in W$ be the $\g k$-invariant vector that is given in Remark %
\ref{rmk-v*}, and set 
\begin{equation*}
\iota_\g a:\g a\to W\ ,\ \iota_\g a(h):=\rho(h)\beta^*. 
\end{equation*}
By duality, the map $\iota_\g a$ results in a homomorphism of superalgebras 
\begin{equation}
\label{dfiotaa*}
\iota_\g a^*:\sP(W)\cong\sS(W^*)\to\sS(\g a^*)\cong \sP(\g a), 
\end{equation}
which is defined uniquely by the relation $\lag \iota^*_\g a (w^*),h\rag =%
\lag w^*,\iota_\g a(h)\rag
$ for $w^*\in W^*$, $h\in\g a$. 

Let $\kappa$ be the supertrace form on $\g g$ defined as in \eqref{kappa}. The restriction of $\kappa$ to $\g a$ is a non-degenerate symmetric
bilinear form and yields a canonical isomorphism 
$\sfj:\g a^*\to\g a$ which is defined as follows.
For every $\xi\in\g a^*$,
\begin{equation*}
\sfj(\xi):=x_\xi\text{ if and only if }\kappa(\,\cdot\,,x_\xi)=\xi .
\end{equation*}
Let $\sfj^*:\sP(\g a)\to\sP(\g a^*)$ be defined by $\sfj^*(p):=p\circ\sfj$ for every $p\in\sP(\g a)$.
For $\lambda\in\mathrm{E}_{m,n}^*$, let $\mathbf{d}_\lambda\in\sP(W)^\g k$ be the 
$\g k$-invariant vector in $V_\lambda$ as in Proposition \ref{le-dlam}.
\begin{dfn}
\label{dfbr}
For  $\lambda\in\mathrm{E}_{m,n,d}^*$,
we define the \emph{spherical polynomial} $d_\lambda$ to be
\begin{equation*}  
d_\lambda:=\sfj^*\circ\iota^*_\g a(\mathbf{d}_\lambda)\in\sP^d(\g a^*).
\end{equation*}
\end{dfn}

\begin{prp}
\label{DGBRVVV}
The restriction of $\iota_\g a^*$ to $\sP(W)^\g k$ is an injection.
In particular, $d_\lambda\neq 0$ for
every $\lambda$.
\end{prp}\begin{proof}
The proof of Proposition \ref{DGBRVVV} 
will be given in Section \ref{sec-pflem}. 
\end{proof}

Recall the definitions of $\widetilde\sfh_\beta$ and $z_{\lambda,\g a}$ from
\eqref{tildhbeta} and \eqref{bfUgan}.
\begin{lem}
\label{lemZnZk=0}
Let $\lambda\in\mathrm{E}^*_{m,n,d}$, and let
$z_{\lambda,\g a}$ be defined as in \eqref{bfUgan}.
Then  
\[d_\lambda=\big(\sfj^*\circ\iota^*_\g a\circ\widetilde\sfh_\beta\circ\widehat \sfs_d\big)(\check\rho(z_{\lambda,\g a})).
\]
\end{lem}
\begin{proof}
Let $D_\lambda$ be the $\g g$-invariant differential operator as in Definition \ref{DefDla}. Since we have $D_\lambda\in\sfm(\sP^d(W)\otimes\sS^d(W))$, from \eqref{eq-shmphh} we obtain
\begin{equation}
\label{formdlda}
\mathbf d_\lambda=
\widetilde\sfh_\beta(\sfm^{-1}(D_\lambda))
=
\widetilde\sfh_\beta(\widehat\sfs_d(D_\lambda))
.\end{equation}
From \eqref{formdlda} it follows that
\[
d_\lambda=\big(\sfj^*\circ\iota^*_\g a\circ\widetilde\sfh_\beta\circ\widehat\sfs_d\big)(\check\rho(z_{\lambda})).
\]
By the decomposition \eqref{bfUgan} 
and the fact that 
$z_{\lambda,\g k},z_{\lambda,\g u^+},z_{\lambda,\g a}\in\bfU^d(\g g)$,
it is enough to prove that 
if $x\in\g k\bfU^{d-1}(\g g)$ or
$x\in\bfU^{d-1}(\g g)\g u^+$, then  
\begin{equation}
\label{iotac1P}
\big(
\iota_\g a^*
\circ
\widetilde\sfh_\beta\circ \widehat \sfs_d\big)(\check\rho(x))
=0.
\end{equation}
First we prove  \eqref{iotac1P} for $x\in\g k\bfU^{d-1}(\g g)$.
By the PBW Theorem there exist $x^\circ,x^-\in\bfU(\g g)$ such that
such that
\begin{itemize}
\item[(i)] $x=x^\circ+x^-$,
\item[(ii)] $x^\circ$ is a sum of monomials of the form $x_1\cdots x_d$ where 
$x_1\in\g k$ and 
$x_2,\ldots,x_d\in\g g$,
\item[(iii)]  $x^{-}\in \bfU^{d-1}(\g g)$.
\end{itemize}
Recall that $\mathrm{ord}(D)$ denotes the order of a differential operator $D\in\sPD(W)$.
Since we have $\mathrm{ord}(\check\rho(x))\leq 1$ for every $x\in \g g$,
from \eqref{whsD1Dk} it follows that
$\widehat\sfs_d(\check\rho(x^-))=0$ 
and
\begin{equation}
\label{wsdrhoc}
\widehat\sfs_d
(\check\rho(x_1\cdots x_d))
=
\widehat\sfs_1(\check\rho(x_1))\cdots
\widehat\sfs_1(\check\rho(x_d)).
\end{equation}
Since $\iota_\g a^*$ and $\widetilde\sfh_\beta$ are
homomorphisms of superalgebras, 
from \eqref{wsdrhoc} it follows that 
in order to prove 
\eqref{iotac1P}, it is enough to verify that
\begin{equation}
\label{pWWcsig1}
\big(
\iota_\g a^*\circ\widetilde\sfh_\beta\circ\widehat\sfs_1
\big)
(\check\rho(x_1))=0
\text{ for every }x_1\in\g k.
\end{equation}
To verify \eqref{pWWcsig1}, we use 
the generators $\{x_{i,j}\,:\,i,j\in\mathcal I_{m,2n}\}$
and
$\{y_{i,j}\,:\,i,j\in\mathcal I_{m,2n}\}$
defined as in \eqref{eq-dfxijyij}.
From \eqref{ppooll} it follows that
\begin{equation}
\label{S1cRho}
\widehat\sfs_1(\check\rho(E_{i,j}))=
-(-1)^{|i|\cdot|j|}
\sum_{r\in\mathcal I_{m,2n}}
(-1)^{|r|}y_{r,j}x_{r,i}\,
\text{ for }i,j\in\mathcal I_{m,2n},
\end{equation}
and consequently,
\begin{equation*}
\big(
\iota_\g a^*\circ\widetilde\sfh_\beta\circ\widehat\sfs_1
\big)
(\check\rho(E_{i,j}))
=
\begin{cases}
-\iota_\g a^*(y_{i,j}) & \text{ if }|i|=|j|=\eev,\\
\iota_\g a^*(y_{\oline{2k},j}) &
\text{ if }
i=\oline{2k-1}\text{ for }1\leq k\leq n,\text{ and }|j|=\ood,\\
-\iota_\g a^*(y_{\oline{2k-1},j}) &
\text{ if }
i=\oline{2k}\text{ for }1\leq k\leq n,\text{ and }|j|=\ood,\\
0& \text{ if }|i|\neq |j|.
\end{cases}
\end{equation*}
If $x_1\in\g k_\ood$, then $x_1$ is a linear combination of the $E_{i,j}$ such that $|i|\neq |j|$, and therefore 
$\big(\iota_\g a^*\circ\widetilde\sfh_\beta\circ\widehat\sfs_1
\big)
(\check\rho(x_1))=0$. If $x_1\in\g k_\eev$, then 
$x_1$ is a linear combination of elements of the 
cases (i)--(iv) in Remark \ref{rmkbasisK},
and in each case, we can verify  that 
$\big(\iota_\g a^*\circ\widetilde\sfh_\beta\circ\widehat\sfs_1
\big)
(\check\rho(x_1))=0$.

The proof of \eqref{iotac1P} for $x\in\bfU(\g g)^{d-1}\g u^+$ is similar. 
As in the case $x\in\g k\bfU(\g g)$, the proof can be reduced to showing that 
\begin{equation}
\label{pWWcsig111}
\big(
\iota_\g a^*\circ\widetilde\sfh_\beta\circ\widehat\sfs_1
\big)
(\check\rho(x))=0
\text{ for every }x\in\g u^+.
\end{equation}
It is easy to verify directly that $\g u^+$ is
spanned by the $E_{i,j}$ for $i,j$ satisfying at least one of the following conditions.
\begin{itemize}
\item[(i)] $1\leq i<j\leq m$, 
\item[(ii)] $i=\oline k$ and $j=\oline l$ where 
$1\leq k< l\leq 2n$ and 
$
(k,l)\not\in\{(2t-1,2t)\,:\,1\leq t\leq n\}
$. 
\end{itemize}
Therefore \eqref{pWWcsig111} follows from the above calculation of 
$\big(\iota_\g a^*\circ\widetilde\sfh_\beta\circ\widehat\sfs_1
\big)
(\check\rho(E_{i,j}))$, together with the fact that
$\iota_\g a^*(y_{i,j})=0$ unless 
$i=j\in\{1,\ldots,m\}$, or $i=\oline{2k-1}$ and $j=\oline{2k}$ where $1\leq k\leq n$, or 
$i=\oline{2k}$ and $j=\oline{2k-1}$ where $1\leq k\leq n$.
\end{proof}

For any polynomial $p\in \sP(\g a^*)$, we write $\oline p$ for the homogeneous part of highest degree of $p$.
Recall that if $\lambda\in\mathrm{E}_{m,n}^*$,
then $c_\lambda$ denotes
the eigenvalue polynomial, as 
in  Definition
\ref{dfcl}, and 
$d_\lambda$ denotes the spherical polynomial, as  in 
Definition
\ref{dfbr}.
We have the following theorem.
\begin{thm}
\label{MAINTHM} For every $\lambda\in\mathrm{E}_{m,n}^*$, we have
 $\oline  c_\lambda=d_{\lambda}$.
\end{thm}

\begin{proof}
Let $\gamma_i$ be the basis of $\g a^*$ defined in \eqref{gamkldf}, and let
$\{h_i\,:\,i\in\mathcal{I}_{m,n}\}$ be the dual basis of $\g a$.
Assume that $\lambda\in\mathrm{E}_{m,n,d}^*$, so that 
$V_\lambda\subset \sP^d(W)$. Recall the definition of $z_{\lambda,\g a}$ from \eqref{bfUgan}.
By the PBW Theorem, we can write
\[
z_{\lambda,\g a}=\sum_{k_1,\ldots,k_m,l_{\oline 1},\ldots,l_{\oline n}\geq 0}
 u^{}_{k_1,\ldots,k_m,l_{\oline 1},\ldots,l_{\oline n}}
h_1^{k_1}\cdots h_m^{k_m}
h_{\oline 1}^{l_{\oline 1}}\cdots 
h_{\oline n}^{l_{\oline n}},
\]
where only finitely many of the scalar coefficients  $u^{}_{k_1,\ldots,k_m,l_{\oline 1},\ldots,l_{\oline n}}
$ are nonzero.
Now consider the polynomial $p_\lambda=p_\lambda(t_1,\ldots,t_{m+n})$ in $m+n$ variables 
$t_1,\ldots,t_{m+n}$, defined by
\[
p_\lambda(t_1,\ldots,t_{m+n}):=\sum_{k_1,\ldots,k_m,l_{\oline 1},\ldots,l_{\oline n}\geq 0}
 u^{}_{k_1,\ldots,k_m,l_{\oline 1},\ldots,l_{\oline n}}
t_1^{k_1}\cdots t_m^{k_m}
t_{m+ 1}^{l_{\oline 1}}\cdots 
t_{m+ n}^{l_{\oline n}},
\]
so that
$
z_{\lambda,\g a}=p_\lambda\left(h_1^{},\ldots,h_m^{},h_{\oline 1},\ldots,h_{\oline n}\right)
$.
We first note that $\deg(p_\lambda)=d$.
Indeed since $z_{\lambda,\g a}\in\bfU^d(\g g)$,  it follows that $\deg(p_\lambda)\leq d$. If
$\deg(p_\lambda)<d$, then $\mathrm{ord}(\check\rho(z_{\lambda,\g a}))<d$, and  Lemma 
\ref{lemZnZk=0} implies that $d_\lambda=0$, which contradicts Proposition \ref{DGBRVVV}.

Fix $\xi:=\sum_{i\in\mathcal I_{m,n}}a_i\gamma_i\in\g a^*$.
%Note that
%$
%\sfj^{-1}(\xi)=\sum_{k=1}^m a_kh_k-\frac{1}{2}\sum_{l=1}^n a_{\oline l} h_{\oline l}
%$.
Lemma \ref{cmulZa} implies that
\begin{align}
\label{cl=pol}
c_\lambda(\xi)
&=p_\lambda(\xi(h_1^{}),\ldots,\xi(h_m^{}),
\xi(h_{\oline 1}),\ldots,\xi(h_{\oline n}))
=p_\lambda(a_1^{},\ldots,a_m^{},a_{\oline 1},\dots,a_{\oline n}).
\end{align}
To complete the proof, it suffices to show that for all $\xi\in\g a^*$ we have
\begin{align}
\label{dlxxi}
d_\lambda(\xi)
=
%\big(\iota_\g a^*\circ \widetilde\sfh_\beta\circ\widehat\sfs_d\big)
%(\check\rho(z_{\lambda,\g a}))(\sfj^{-1}(\xi))
%=
\oline p_\lambda
\big(a_1^{},\ldots,a_m^{},a_{\oline 1},\ldots,a_{\oline n}\big),
\end{align}
where $\oline p_\lambda$ denotes the homogeneous part of
highest degree of $p_\lambda$.
Set $\mathbf D_i:=\check\rho(h_i)\in\sPD(W)$ for every
$i\in\mathcal I_{m,n}$. 
Since $\g a$ is commutative, we have $\mathbf D_i\mathbf D_j=\mathbf D_j\mathbf D_i$ for every $i,j\in\mathcal I_{m,n}$, and thus
\[
\check\rho(z_{\lambda,\g a})
=
p_\lambda(\mathbf D_1,\ldots,\mathbf D_m,\mathbf D_{\oline 1},\ldots,\mathbf D_{\oline n}).
\]
The last relation together with the fact that
$\deg(p_\lambda)=d$ imply that
\begin{align*}
\widehat\sfs_d
(\check\rho(z_{\lambda,\g a}))&=
{\oline p}_\lambda
\big(\widehat\sfs_1(\mathbf D_1),
\ldots,
\widehat\sfs_m(\mathbf D_m),
\widehat\sfs_1(\mathbf D_{\oline{1}}),
\ldots,
\widehat\sfs_1(\mathbf D_{\oline{n}})
\big),
\end{align*}
Using \eqref{S1cRho} we obtain
\begin{equation}
\label{bigyehhhh}
\widetilde\sfh_\beta(\widehat\sfs_d
(\check\rho(z_{\lambda,\g a})))
=
\oline p_\lambda\big(-y_{1,1},\ldots,-y_{m,m},
-2y_{\oline 1,\oline 2},\ldots,-2y_{\oline{2n-1},\oline {2n}}\big).
\end{equation}
It is straightforward to verify that 
\[
\lag \iota_\g a^*(y_{k,k}),\sfj(\xi)\rag =-a_k\,\text{ for }
1\leq k\leq m
\ \text{ and }\ 
\lag \iota_\g a^*(y_{\oline{2l-1},\oline{2l}}),\sfj(\xi)\rag=-\frac12 a_{\oline l}
\,\text{ for }
1\leq l\leq n.
\]
Thus, by composing both sides of 
\eqref{bigyehhhh} with 
$\sfj^*\circ \iota_\g a^*$
and then evaluating both sides at $\xi$, we obtain
\begin{align*}
%\label{dlxxi}
d_\lambda(\xi)
=
\big(\sfj^*\circ \iota_\g a^*\circ \widetilde\sfh_\beta\circ\widehat\sfs_d\big)
(\check\rho(z_{\lambda,\g a}))(\xi)
=
\oline p_\lambda
\big(a_1^{},\ldots,a_m^{},a_{\oline 1},\ldots,a_{\oline n}\big).
\end{align*}
This establishes \eqref{dlxxi} and completes the proof.
\end{proof}

Our final result in this section is a characterization of the eigenvalue polynomial $c_\lambda$ by its symmetry and vanishing properties. 
Let 
\begin{equation}
\label{HC+dfneq}
\HC^+:\bfU(\g g)\to \bfU(\g a)
\end{equation}
denote the Harish--Chandra projection
defined by the composition 
\begin{equation*}
\bfU(\g g) \xrightarrow{\ \psf^+\ }\bfU(\g h) \xrightarrow{\ \sfq\ }\bfU(\g %
a) 
\end{equation*}
where $\psf^+:\bfU(\g g)\to\bfU(\g h)$ is the projection according to the
decomposition 
\begin{equation*}
\bfU(\g g)=\left(\bfU(\g g)\g n^+ +\g n^-\bfU(\g g)\right)\oplus\bfU(\g h) 
\end{equation*}
and $\sfq:\bfU(\g h)\to\bfU(\g a)$ is the projection corresponding to the
decomposition $\g h=\g a\oplus\g t$.
For every $\lambda\in\mathrm{E}_{m,n}^*$, let
$z_\lambda\in\bfZ(\g g)$ be defined as in 
\eqref{choicezl}. A precise description of the algebra $\psf^+(\bfZ(\g g))\subseteq \bfU(\g h)$ is 
 known
(see for example \cite{sergeev82}, \cite{sergeev3}, \cite{KacZ}, \cite{Gorelik}, or \cite[Sec. 2.2.3]{ChWabook}), 
%Let $\widetilde h_i:=E_{i,i}$for $i\in\mathcal I_{m,2n}$ are the standard basis elements of $\g h$.  Then 
and implies  that  
$\psf^+(\bfZ(\g g))$ is equal to the subalgebra of 
$\bfU(\g h)$ generated by the elements
\begin{equation}
\label{dfnGd}
G_d:=\sum_{k=1}^m\left(
E_{k,k}+\frac{m+1}{2}-n-k
\right)^d
+(-1)^{d-1}
\sum_{l=1}^{2n}
\left(
E_{\oline{l},\oline{l}}+\frac{m+1}{2}+n-l
\right)^d
\end{equation}
for every $d\geq 1$.
%By the canonical isomorphism $\bfU(\g h)\cong\sS(\g h)\cong\sP(\g h^*)$, we can identify $\psf^+(\bfZ(\g g))$ with a subalgebra $\mathbf I(\g h^*)\subseteq\sP(\g h^*)$.
Let $\mathbf I(\g a^*)\subseteq\sP(\g a^*)$ be the subalgebra that corresponds to $\sfq(\psf^+(\bfZ(\g g))$ via the canonical isomorphism $\bfU(\g a)\cong\sS(\g a)\cong \sP(\g a^*)$.

\begin{thm}
\label{thm-unqclam}
Let $\lambda\in\mathrm{E}_{m,n,d}^*$. 
Then ${c}_\lambda$ is 
the unique element of 
$
\mathbf I(\g a^*)
$
that satisfies
$\deg(c_\lambda)\leq d$, $c_\lambda(\lambda)=d!$,
and
 $c_\lambda(\mu)=0$ for all other
$\mu$ in $\bigcup_{d'=0}^d \mathrm{E}_{m,n,d'}^*
$.
\end{thm}

\begin{proof}
By 
Remark \ref{rmkdgclam} and
Lemma \ref{lem-clmuvanish}, only the uniqueness statement requires proof.\\

\noindent\textbf{Step 1.}
We prove that \begin{equation}
\label{eqIa*=}
\mathbf I(\g a^*)\cap
\bigoplus_{d'=0}^d\sP^{d'}(\g a^*)=\mathrm{Span}_\C
\left\{ c_\lambda\,:\,
\lambda\in\bigcup_{d'=0}^d\mathrm{E}_{m,n,d'}^*
\right\}
.
\end{equation}
Fix $z\in\bfZ(\g g)$ and let $\mu\in\mathrm{E}_{m,n}^*$. Let $v_\mu\in V_\mu$ be a 
highest weight vector and let $v_\circ^*\in V_\mu^*$ be a nonzero $\g k$-invariant vector. As shown in the proof of Lemma \ref{cmulZa}, we have $\lag v_\circ^*,v_\mu\rag\neq 0$. Since
$z\in\bfZ(\g g)$, we have
 $z-\psf^+(z)\in\bfU(\g g)\g n^+\cap\g n^-\bfU(\g g)$, and therefore
\begin{align*}
\lag v_\circ^*,\check\rho(z)v_\mu\rag
&=
\lag v_\circ^*,\check\rho(\psf^+(z))v_\mu\rag\\
&=
\lag v_\circ^*,\check\rho(\sfq(\psf^+(z)))v_\mu\rag=
\mu(\sfq(\psf^+(z)))\lag v_\circ^*,v_\mu\rag
=
\mu(\HC^+(z))\lag v_\circ^*,v_\mu\rag.
\end{align*}
Since $\check\rho(z)\in\sPD(W)^\g g$, we can write $\check\rho(z)$ as a linear combination of Capelli operators (see Definition \ref{DefDla}), say
$\check\rho(z)=\sum_k{a_k}D_{\lambda_k}$.
It follows that the map $\mu\mapsto\mu(\HC^+(z))$ agrees with  $\sum_k a_k c_{\lambda_k}(\mu)$ for $\mu\in\mathrm{E}_{m,n}^*$. Since $\mathrm{E}_{m,n}^*$ is Zariski dense in $\g a^*$,  the latter two maps should agree for all  $\mu\in\g a^*$ as well. This implies 
that
\[
\mathbf I(\g a^*)\subseteq\mathrm{Span}_\C\{c_\lambda\,:\,\lambda\in\mathrm{E}_{m,n}^*\}.
\] 
Furthermore, 
from Theorem \ref{MAINTHM} and Proposition \ref{DGBRVVV} it follows that the homogeneous part of highest degree of every nonzero element of  \[
\mathrm{Span}_\C\left\{c_\lambda\,:\,\lambda\in\bigcup_{d'>d}\mathrm{E}_{m,n,d'}^*\right\}
\]
has degree strictly bigger than $d$. 
Consequently, the left hand side of \eqref{eqIa*=} is a subset of its right hand side. The reverse inclusion follows from the above argument by choosing $z:=z_\lambda$ where $z_\lambda$ is defined in 
\eqref{choicezl}.\\

\noindent\textbf{Step 2.} 
Set $N_{m,n,d}:=\sum_{d'=0}^d\big|\mathrm{E}_{m,n,d'}^*\big|$.
From Step 1 it follows that
\[
\dim\left(
\mathbf I(\g a^*)\cap
\bigoplus_{d'=0}^d\sP^{d'}(\g a^*)
\right)
\leq N_{m,n,d}.
\] Now consider the linear map
\[
L_{m,n,d}:\mathbf I(\g a^*)\cap
\bigoplus_{d'=0}^d\sP^{d'}(\g a^*)\to\C^{N_{m,n,d}}
\ ,\
p\mapsto \big(p(\mu)\big)_{\mu\in\bigcup_{d'=0}^d
\mathrm{E}_{m,n,d'}^*}.
\]
From Lemma \ref{lem-clmuvanish} it follows that the vectors $L_{m,n,d}(c_\lambda)$, for $\lambda\in\bigcup_{d'=0}^d\mathrm{E}_{m,n,d'}^*$,  form a 
nonsingular triangular matrix, so that they form a basis of $\C^{N_{m,n,d}}$. Consequently,  $L_{m,n,d}$ is an invertible linear transformation. In particular, 
for every $\lambda\in\mathrm{E}_{m,n,d}^*$, the 
polynomial $c_\lambda$
is the unique element of $\mathbf I(\g a^*)\cap
\bigoplus_{d'=0}^d\sP^{d'}(\g a^*)$  that satisfies
$c_\lambda(\lambda)=d!$
and
 $c_\lambda(\mu)=0$ for all other
$\mu\in\bigcup_{d'=0}^d \mathrm{E}_{m,n,d'}^*$.
\end{proof}

\section{Relation with Sergeev--Veselov polynomials for $\theta=\frac{1}{2}$}

\label{SecRelSer}

The main result of this section is 
Theorem \ref{thmconnSV},
which describes the precise relation between the eigenvalue polynomials $
c_\lambda$ of Definition \ref{dfcl}
and the shifted super Jack polynomials of Sergeev--Veselov \cite%
{SerVes}. 
As in the previous section, we set $V:=\C^{m|2n}$, so that 
$\g g:=\gl(V)=\gl(m|2n)$. The map $\omega:\g g\to\g g$ given by $%
\omega(x)=-x$ is an anti-automorphism of $\g g$, that is, 
\[
\omega([x,y])=(-1)^{|x|\cdot|y|}[\omega(y),\omega(x)]
\] for $x,y\in\g g$.
Therefore $\omega$ extends canonically to an anti-automorphism $\omega:\bfU(%
\g g)\to\bfU(\g g)$, that is, 
\begin{equation*}
\omega(xy)=(-1)^{|x|\cdot|y|}\omega(y)\omega(x)  
\ \text{ for }x,y\in\bfU(\g g).
\end{equation*}

Recall that $\HC^+:
\bfU(\g g)\to \bfU
(\g a)$
is the Harish--Chandra
projection as in \eqref{HC+dfneq}.
Let \[
\HC^-:\bfU(\g g)\to \bfU%
(\g a)
\] be the opposite Harish--Chandra projection defined by the
composition 
\begin{equation*}
\bfU(\g g) \xrightarrow{\ \psf^-\ }\bfU(\g h) \xrightarrow{\ \sfq\ }\bfU(\g %
a) 
\end{equation*}
where where $\psf^-:\bfU(\g g)\to\bfU(\g h)$ is the projection according to
the decomposition 
\begin{equation*}
\bfU(\g g)=\left(\bfU(\g g)\g n^- +\g n^+\bfU(\g g)\right)\oplus\bfU(\g h). 
\end{equation*}
It is straightforward to verify that 
\begin{equation}  \label{HC+HC-p}
\HC^-(\omega(z))=\omega(\HC^+(z))\text{ for }z\in\bfU(\g g).
\end{equation}
 
%\begin{lem}
%\label{genofHC+}
%Let $z\in\bfZ(\g g)$. Then $\HC^+(z)$ belongs to the unital subalgebra of $\bfU(\g a)$ that is generated by elements of 
%the form
%\begin{align*}
%\sum_{i=1}^m&
%\left(h_i+\frac{m+1}{2}-n-i\right)^d
%\\
%&+(-1)^{d-1}
%\sum_{j=1}^n
%\left(
%\frac12 h_{\oline j}+\frac{m+1}{2}+n-2j+1
%\right)^d
%+
%\left(
%\frac12 h_{\oline j}+\frac{m+1}{2}+n-2j
%\right)^d
%\end{align*}
%for every $d\geq 1$.
%\end{lem}
%\begin{proof}
%Follows by a straightforward calculation from the %Harish--Chandra homomorphism 
%of Kac and Gorelik 
%\cite[Sec. 2.2.3]{ChWabook}.
%\end{proof}

\begin{dfn}
\label{Dfndualmu}
For every $\mu\in\mathrm{E}_{m,n}^*$, we use $\mu^*$ to denote the unique element 
of $\mathrm{E}_{m,n}$
that satisfies $V_\mu^*\cong V_{\mu^*}^{}$. 
\end{dfn}
Using formulas \eqref{frmLamBk} and \eqref{prpSWW*}, one can see that the map $\mu\mapsto \mu^*$ is not linear. Nevertheless, the following proposition still holds.
\begin{prp}
\label{prpQlam}
Let $\lambda\in\mathrm{E}_{m,n,d}^*$. Then there exists a unique polynomial $c_\lambda^*\in\sP(\g a^*)$ such that  
$\deg(c_\lambda^*)\leq d$ and $
c_\lambda(\mu)=c_\lambda^*(\mu^*)
$ for every 
$\mu\in\mathrm{E}_{m,n}^*$.
\end{prp}

\begin{proof}
First we prove the existence of $c^*_\lambda$.
Recall that the action of $D_\lambda:V_\mu\to V_\mu$ is by the scalar $c_\lambda(\mu)$. 
%Furthermore, $V_\mu^{}\cong (V_{\mu^*})^*$ as $\g g$-modules. 
Let $v_{\mu^*}$ denote the highest weight of $V_{\mu^*}$, and 
define
$v_\mu^-\in (V_{\mu^*})^*\cong V_\mu$ by
 $\lag v_\mu^-,v_\mu^{*}\rag =1$ and $\lag v_\mu^-,\bfU(\g n^-)v_{\mu^*}\rag=0$. It is straightforward to verify that $v_\mu^-\in (V_\mu)^{\g n^-}$, and thus
 $v_{\mu}^-$ is the lowest weight vector of $V_\mu^{}$. It follows immediately that the lowest weight of $V_\mu^{}$ is $-\mu^*$. 
 Choose $z_\lambda\in\bfZ(\g g)$ as in \eqref{choicezl}. 
By considering the $\g h$-action on $\bfU(\g g)$ we obtain $
z_\lambda-\psf^-(z_\lambda)\in\bfU(\g g)\g n^-
$, and therefore
\[
c_\lambda(\mu)v_\mu^-=D_\lambda v_\mu^-=\check\rho(z_\lambda)v_\mu^-
=\check\rho(\psf^-(z_\lambda))v_\mu^-=
\sfh_{-\mu^*}(\HC^-(z_\lambda))v_\mu^-,
\]
where $\sfh_{-\mu^*}:\bfU(\g a)\cong\sS(\g a)\to\C$ is defined as in \eqref{dfheta}.
It is straightforward to check that 
the map 
\begin{equation}
\label{eq-dfQlamb}
c_\lambda^*(\nu):=
\sfh_{-\nu}(\HC^-(z_\lambda))
\end{equation} is a polynomial in $\nu\in\g a^*$. 
From \eqref{HC+HC-p} and Theorem \ref{prpgw} it follows that
$\deg(c^*_\lambda)\leq d$.
Finally, uniqueness of $c^*_\lambda$ follows from the fact that 
$\mathrm{E}_{m,n}$ is Zariski dense in $\g a^*$. 
 \end{proof}
 
We now recall 
the definition of
the algebra of \emph{shifted
supersymmetric polynomials} 
\[
\Lambda_{m,n,\frac12}^\natural\subseteq\sP(\C^{m+n}),
\]
introduced in \cite[Sec. 6]{SerVes}. 
% the shifted super Jack polynomials $SP^*_\flat$ from \cite{SerVes}. 
%Let \[
%\Lambda_{m,n,\frac12}^\natural\subseteq\sP(\C^{m+n})
%\] be the algebra of \emph{shifted
%supersymmetric polynomials} introduced in \cite[Sec. 6]{SerVes}. 
For $1\leq k\leq r$, let $\mathsf{e}_{k,r}$ be 
the $k$-th unit vector in $\C^r$. 
Then the algebra 
$\Lambda_{m,n,\frac12}^\natural$
consists of polynomials $f(\sfx_1,\ldots,\sfx_m,\sfy_1,\ldots,\sfy_n)$, 
which are 
separately 
symmetric in $\sfx:=(\sfx_1,\ldots,\sfx_m)$ and
in $\sfy:=(\sfy_1,\ldots,\sfy_n)$, and which satisfy
the relation 
\begin{equation*}
\textstyle f(\sfx+\frac{1}{2}\mathsf{e}_{k,m},\sfy-\frac{1}{2}\mathsf{e}_{l,n})=f(\sfx-\frac{1}{2}\mathsf{e}_{k,m},\sfy+\frac{1}{2}\mathsf{e}_{l,n}) 
\end{equation*}
on every hyperplane $\sfx_k+\frac12 \sfy_l=0$, where $1\leq k\leq m$ and $1\leq
l\leq n$.

In \cite[Sec. 6]{SerVes},  Sergeev and Veselov introduce a  basis 
 of 
$\Lambda_{m,n,\frac12}^\natural$,
\begin{equation}
\label{SPbinLL}
\left\{
SP_\flat^*\in\Lambda_{m,n,\frac12}^\natural\ :\
\flat\in\mathrm{H}_{m,n}\right\},
\end{equation}
indexed by the set $\mathrm{H}_{m,n}$  of $(m|n)$-hook partitions (see
Definition \ref{dfnhookprn}).
The polynomials $SP_\flat^*$ are called \emph{shifted super Jack polynomials} and they satisfy certain vanishing conditions, given in \cite[Eq. (31)]{SerVes}.

Recall the map 
$\boldsymbol\Gamma:\mathrm{H}_{m,n}\to \g a^*$ 
from \eqref{bfGam}.
Given $\flat=(\flat_1,\flat_2,\flat_3,\ldots)\in\mathrm{H}_{m,n}$, we set $%
\flat_k^*:=\max\{\flat_k^{\prime }-m,0\}$
for every $k\geq 1$.
 Fix $\mu\in\mathrm{E}_{m,n}^*$,
and let $\mu^*\in \mathrm{E}_{m,n}$ be  as in Definition
\ref{Dfndualmu}. From the
definition of the map $\boldsymbol{\Gamma}$ it follows that there exists a $\flat\in\mathrm{H}_{m,n}$ such that
\begin{equation}
\label{fordfGAM}
\mu^*=\boldsymbol{\Gamma}(\flat)=
\sum_{i=1}^m2\flat_i\gamma_i+\sum_{j=1}^n2\flat_j^* \gamma_{\oline j}. 
\end{equation}
Now let $\mathsf{F}:\mathrm{H}_{m,n}\to \C^{m+n}$ be the \emph{Frobenius map} of \cite%
[Sec. 6]{SerVes}, defined as follows.
For every $\flat\in\mathrm{H}_{m,n}$, 
we have
$\mathsf{F}(\flat)=\left(\sfx_1(\flat),\ldots,\sfx_m(\flat),\sfy_1(\flat),\ldots,\sfy_n(\flat)\right)$, where
 
\begin{equation}
\label{xkfykfl}
\begin{cases}
\sfx_k(\flat):=\flat_k-\frac12(k-\frac12)
-\frac12(2n-\frac m2) & \text{ for }1\leq k\leq m, \\ 
\sfy_l(\flat):=\flat_l^*-2(l-\frac12)+\frac12(4n+m) & \text{ for }1\leq l\leq n.%
\end{cases}
\end{equation}
%Given a polynomial $Q\in\sP(\g a^*)$, we define
%the \emph{Frobenius transform}
%\begin{equation}
%\label{FrobT}
%\mathscr{F}(Q)\in \sP(\C^{m+n})
%\end{equation}
%as follows. For every $\mathsf{v}\in\mathsf{F}%(\mathrm{H}_{m,n})$ we set
%\[
%\mathscr{F}(Q)(\mathsf v):=Q(\boldsymbol\Gamma(\mathsf{F}^{-1}(\mathsf{v}))).
%\]
\begin{lem}
\label{lem6..3}
The map
$\boldsymbol\Gamma\circ\mathsf{F}^{-1}$ defined initially on 
$\mathsf{F}(\mathrm{H}_{m,n})$ 
extends uniquely to an affine linear map $\Psi:\C^{m+n}\to \g a^*$.
\end{lem}
\begin{proof}
This is immediate from 
the formulas 
\eqref{fordfGAM} and
\eqref{xkfykfl}, and the Zariski density of  
$\mathsf{F}(\mathrm{H}_{m,n})$ in 
$\C^{m+n}$.
\end{proof}

% and 
% is indeed given by an affine change of coordinates, so that
\begin{dfn}
\label{FrobT}
For $p\in\sP(\g a^*)$, we define its \emph{Frobenius transform} to be
\[
\mathscr{F}(p):=p\circ\Psi.
\]
\end{dfn} 
Note that \begin{equation}
\label{degpdegFpeq}
\deg(p)=\deg(\mathscr{F}(p))
\text{ for every }p\in\sP(\g a^*).
\end{equation}
We now relate the shifted super Jack polynomials $SP_\flat^*$ to the dualized eigenvalue polynomials
$c_\lambda^*$ from Proposition \ref{prpQlam}.
\begin{thm}
\label{thmconnSV}
Let $\lambda\in\mathrm{E}_{m,n,d}^*$
and let $\flat\in\mathrm{H}_{m,n,d}$ be such that $\lambda^*=\boldsymbol\Gamma(\flat)$. Then 
we have 
\begin{equation}
\label{Ql=GGSP*}
\mathscr{F}(c_\lambda^*)=
\frac{d!}{H(\flat)}SP^*_\flat,
\end{equation}
where 
$
H(\flat)
:=\prod_{k\geq 1}\prod_{l=1}^{\flat_k}
\left(
\flat_k-l+1+\frac12(\flat_l'-k)
\right)
$.

\end{thm}

\begin{proof}
By 
\eqref{degpdegFpeq},
Proposition \ref{prpQlam},
 and the results of \cite[Sec. 6]{SerVes},
both sides of \eqref{Ql=GGSP*} are in
$\bigoplus_{d'=0}^d\sP^{d'}(\C^{m+n})$.
Next we prove that \begin{equation}
\label{FQinL}
\mathscr{F}(c^*_\lambda)
\in
\Lambda^\natural_{m,n,\frac12}\,. 
\end{equation}
Note that $\bfZ(\g g)$ is invariant under the anti-automorphism $\omega$.
Let $z_\lambda\in\bfZ(\g g)\cap\bfU^d(\g g)$ be 
as in \eqref{choicezl}. 
From \eqref{HC+HC-p}
 we obtain 
 $\HC^-(z_\lambda)=\omega(\HC^+(\omega(z_\lambda)))$, and thus 
%Lemma \ref{genofHC+} implies that 
$\HC^-(z_\lambda)$ is in the subalgebra 
of $\bfU(\g a)\cong\sS(\g a)$ that is 
generated by $\omega(\sfq(G_d))$ for $d\geq 1$, where 
$G_d$ is given in \eqref{dfnGd}. Observe that
$\sfq(G_d)$ is obtained from $G_d$ by the substitutions 
\[
E_{k,k}\mapsto h_k\text{ for }1\leq k\leq m
\ \text{ and }\
E_{\oline{2l-1},\oline{2l-1}},E_{\oline{2l},\oline{2l}}\mapsto \frac{1}{2}h_{\oline l}
\text{ for }
1\leq l\leq n,
\]
where
$\{h_i\,:\,i\in\mathcal{I}_{m,n}\}$ is the basis for $\g a$ that
is dual to the $\g a^*$-basis $\{\gamma_i\,:\,i\in\mathcal{I}_{m,n}\}$, defined in
\eqref{gamkldf}.
%elements of the form 
%\begin{align*}
%\sum_{k=1}^m&
%\left(-h_k+\frac{m+1}{2}-n-k\right)^d
%\\
%&+(-1)^{d-1}
%\sum_{l=1}^n
%\left(
%-\frac12 h_{\oline l}+\frac{m+1}{2}+n-2l+1
%\right)^d
%+
%\left(
%-\frac12 h_{\oline l}+\frac{m+1}{2}+n-2l
%\right)^d
%\end{align*}
%for every $d\geq 1$. 
By a straightforward calculation we can verify that
 $\mathscr{F}(c^*_\lambda)$ is in the algebra generated by the polynomials
\begin{equation}
\label{2x+ny-n}
\sum_{k=1}^m(2\sfx_k+n)^d+
(-1)^{d-1}\sum_{l=1}^n\left(\sfy_l-n+\frac12\right)^d+\left(\sfy_l-n-\frac12\right)^d
\end{equation}
for every $d\geq 1$. Furthermore, the polynomials 
\eqref{2x+ny-n} belong to $\Lambda_{m,n,\frac12}^\natural$. 
This completes the proof of 
\eqref{FQinL}.

Next we fix $\mu\in\mathrm{E}_{m,n}^*$
and choose 
$\flat^{(\mu)}\in \mathrm{H}_{m,n}$
 such that
$\mu^*=\boldsymbol\Gamma(\flat^{(\mu)})$,
where $\mu^*\in\mathrm{E}_{m,n}$ is
 as in Definition \ref{Dfndualmu}. 
Then
\[
\textstyle
\mathscr{F}(c^*_\lambda)
\left(\mathsf{F}\left(\flat^{(\mu)}\right)\right)
=
c^*_\lambda(\mu^*)=c_\lambda(\mu)
.
\]
From \cite[Eq. (31)]{SerVes} we have
$SP^*_\flat\left(\mathsf{F}\left(\flat^{(\mu)}\right)
\right)=H\left(\flat^{(\mu)}\right)$ for $\mu=\lambda$, and 
$SP_\flat^*\left(\mathsf{F}\left(\flat^{(\mu)}\right)\right)=0$
for all other $\mu\in\bigcup_{d'=0}^d\mathrm{E}_{m,n,d'}^*$.
The equality \eqref{Ql=GGSP*} now follows from
the latter vanishing property and 
Theorem \ref{thm-unqclam},
or alternatively from the discussion immediately above \cite[Eq. (31)]{SerVes}.
\end{proof}
\vspace{2mm}

\begin{center}
\textbf{\large Appendix}
\end{center}

\begin{alphasection}

\section{Proof of Proposition \ref{DGBRVVV}}
\label{sec-pflem}
In this appendix we prove Proposition \ref{DGBRVVV}. The proof of this proposition is similar
to the proof of Proposition \ref{le-dlam}, although somewhat more elaborate.
Recall the generators
$\{x_{i,j}\}_{i,j\in\mathcal I_{m,2n}}$ for 
$\sS(W)$ 
and
$\{y_{i,j}\}_{i,j\in\mathcal I_{m,2n}}$ for $\sP(W)$, defined in
\eqref{eq-dfxijyij}. 
We consider the total ordering $\prec$ on $\mathcal I_{m,2n}$ given by
\[
1\prec\cdots\prec m\prec \oline 1\prec\cdots\prec \oline{2n}.
\]
Set
$
\mathcal I':=\left\{\big(\oline{2k-1},
\oline{2k}\big)\,:\,1\leq k\leq n\right\}
$
and
$\mathcal I'':=
\left\{
(i,j)\in\mathcal I_{m,2n}\times \mathcal I_{m,2n}\ ,\ 
i\prec j
\right\}
\setminus
\mathcal I'
$.
For every $a_1,\ldots,a_m,a_{\oline{1}},\ldots,a_{\oline{n}}\in\C$,
we set 
\begin{equation}
\label{dfbfx}
\bfx:=-\sum_{k=1}^m a_k h_k+\sum_{l=1}^n a_{\oline{l}}h_{\oline{l}}
\in\g a,
\end{equation} where $\{h_i\,:\,i\in\mathcal I_{m,n}\}$ is the basis of $\g a$ that is dual to $\{\gamma_i\,:\,i\in\mathcal I_{m,n}\}$
defined in \eqref{gamkldf}. Let 
$
\sfh_{\bfx}:\sP(\g a)\cong\sS(\g a^*)\to \C
$ be as defined  in 
\eqref{dfheta}.
For any $\bfx\in\g a$ as above, set
\[
\xi_{\bfx}:=
\frac12\sum_{k=1}^ma_kx_{k,k}+\sum_{l=1}^n a_{\oline l}x_{\oline{2l-1},\oline{2l}}\in W,
\]
and let  
$
\sfh_{\xi_{\bfx}}:\sP(W)\cong\sS(W^*)\to \C
$
be as defined  in \eqref{dfheta}.
Observe that $\sfh_{\bfx}\circ\iota_\g a^*=\sfh_{\xi_\bfx}$ where
$\iota_\g a^*$ is as defined in \eqref{dfiotaa*}.
In particular, for every $a\in\sP(W)$,  the set 
\[
S_a:=\left\{\bfx\in\g a\ :\ \sfh_{\xi_{\bfx}}(a)=0\right\}
\]
is Zariski closed in $\g a$.

Let $\partial_{i,j}:=\partial_{x_{i,j}}$ denote the superderivation of 
$\sP(W)$ that is defined according to \eqref{pww*}.
\begin{prp}
Let $\mathbf d\in\sP(W)^\g k$
such that 
$\sfh_{\xi_{\bfx}}(\mathbf d)= 0$
 for every 
$\bfx\in\g a$. Then
\begin{equation} 
\label{hvcthht}
\sfh_{\xi_{\bfx}}
\big(
\partial_{{i_1,j_1}}\cdots\partial_{{i_s,j_s}}
\mathbf d\big)=0\
\text{ for all }
\bfx\in\g a,\ s\geq 1,\ 
\text{and }(i_1,j_1),\ldots,(i_s,j_s)\in\mathcal I''.
\end{equation}

\end{prp}

\begin{proof}
We use induction on $s$. 
%It is similar to the proof of Proposition \ref{le-dlam}. 
First assume that $s=1$. 
Then
\begin{equation}
\label{chrRsph}
\check\rho(x)\mathbf d\text{ for all }x\in\g k.
\end{equation}
Set $x:=E_{k,l}-E_{l,k}$ for $1\leq k<l\leq m$. Then
\eqref{chrRsph} and \eqref{ppooll} imply that
\begin{align*}
0&=\sfh_{\xi_{\bfx}}(\check\rho(x)\mathbf d)\\
&=-\sum_{r\in\mathcal I_{m,2n}}
(-1)^{|r|}\sfh_{\xi_{\bfx}}(y_{r,l})
\sfh_{\xi_{\bfx}}(\partial_{{r,k}}\mathbf d)
+
\sum_{r\in\mathcal I_{m,2n}}
(-1)^{|r|}\sfh_{\xi_{\bfx}}(y_{r,k})
\sfh_{\xi_{\bfx}}(\partial_{{r,l}}\mathbf d)\\
&=(-a_l+a_k)\sfh_{\xi_{\bfx}}(\partial_{{k,l}}\mathbf d),
\end{align*}
from which it follows that 
$\sfh_{\xi_{\bfx}}(\partial_{{k,l}}d_\lambda)=0$ for all $\bfx\in\g a$
as in \eqref{dfbfx} which satisfy
$a_k\neq a_l$.
But the set of all $\bfx\in\g a$ given as in \eqref{dfbfx} which satisfy $a_k\neq a_l$ for all $1\leq k<l\leq m$ is a Zariski dense subset 
of $\g a$, and it follows that 
$\sfh_{\xi_{\bfx}}(\partial_{x_{k,l}}\mathbf d)=0$ for every $\bfx\in\g a$.
A similar argument for each of the cases (ii)--(vi) of 
Remark \ref{rmkbasisK} (where in cases (ii)--(iv) we assume $k\neq l$) proves \eqref{hvcthht} for $s=1$.

Next 
we define an involution
$i\mapsto i^\dagger$ on $\mathcal I_{m,2n}$
by
\[ 
k^\dagger:=k\text{ for }1\leq k\leq m,\ 
(\oline{2l-1})^\dagger:=\oline{2l}\text{ for }1\leq l\leq n,\text{ and }
(\oline{2l})^\dagger:=\oline{2l-1}\text{ for }1\leq l\leq n.
\]
Let $\mathsf f:\mathcal I_{m,2n}\to\mathcal I_{m,n}$
be  defined by $\mathsf f(k)=k$ for $1\leq k\leq m$ and $\mathsf f\big(\oline{2l-1}\big)
=\mathsf f\big(\oline{2l}\big)=\oline{l}$ for $1\leq l\leq n$.
Fix an element $x\in\g k$ that belongs to the spanning set of $\g k$  given in Remark \ref{rmkbasisK}. (When $x$ is chosen from one of the cases (ii)--(iv)
in Remark \ref{rmkbasisK},
we assume that $k\neq l$.) Then 
there exist $p,q\in\mathcal I_{m,2n}$ such that
\begin{equation}
\label{pprecqqpre}
p\prec q,\ p^\dagger\prec q^\dagger,\ \text{and } 
\check\rho(x)=\sum_{r\in\mathcal I_{m,2n}}
\left(\pm y_{r,p}^{}\partial_{{r,q}^{}}
\pm y_{r,q^\dagger_{}}
\partial_{{r,p^\dagger}}\right).
\end{equation}
Now fix $x_1,\ldots,x_{s+1}\in\g k$ such that every
$x_k$, for $1\leq k\leq s+1$, 
is an element of the spanning set of $\g k$ given in Remark \ref{rmkbasisK}.  For every $1\leq k\leq s+1$, if $x_k$ is chosen from  the cases (ii)--(iv)
in Remark \ref{rmkbasisK},
then we assume that $k\neq l$. Choose 
$(p_1,q_1),\ldots,(p_{s+1},q_{s+1})\in\mathcal I_{m,2n}\times\mathcal I_{m,2n}$ corresponding to 
$x_1,\ldots,x_{s+1}$ which satisfy
\eqref{pprecqqpre}.
For $1\leq u\leq s+1$, we define
\[
p_{u,S}:=\begin{cases}
p_u& \text{ if }u\in S,\\
(q_u)^\dagger&\text{ if }u\not\in S,
\end{cases}
\ \ \text{ and }\ \
q_{u,S}:=\begin{cases}
q_u& \text{ if }u\in S,\\
(p_u)^\dagger&\text{ if }u\not\in S.
\end{cases}
\]
Then
\begin{align}
\label{x1xs+1s}
\check\rho&(x_1)\cdots\check\rho(x_{s+1})\mathbf d
\\
&=\sum_{S\subseteq \{1,\ldots,s+1\}}\
\sum_{r_1,\ldots,r_{s+1}\in\mathcal I_{m,2n}}
\left(\pm y_{r_1,p_{1,S}}^{}\cdots y_{r_{s+1},p_{s+1,S}}^{}
\partial_{{r_1,q_{1,S}}^{}}^{}\cdots
\partial_{{r_{s+1},q_{s+1,S}}^{}}(\mathbf d)\right)+R_{\mathbf d},
\notag
\end{align}
where 
$R_{\mathbf d}$ is a sum of terms of the form
$b\partial_{p_1',q_1'}\cdots \partial_{p_{t}', q_{t}'}(\mathbf d)$, 
with $b\in\sP(W)$ and $t\leq s$.
By the induction hypothesis, $\sfh_{\xi_\bfx}(R_{\mathbf d})=0$ for every $\bfx\in\g a$. 
Thus \eqref{x1xs+1s} implies that
\begin{align}
\label{Ssub1s++1}
0&=\sfh_{\xi_\bfx}
(\check\rho(x_1)\cdots\check\rho(x_{s+1})\mathbf d)\\
&=
\sum_{S\subseteq \{1,\ldots,s+1\}}
\pm
a_{\mathsf f(p_{1,S}^{})}^{}\cdots
a_{\mathsf f(p_{s+1,S}^{})}^{}
\sfh_{\xi_\bfx}\left(
\partial_{{(p_1)^\dagger,q_1^{}}}
\cdots
\partial_{{(p_{s+1})^\dagger,q_{s+1}^{}}}
\mathbf d\right).
\notag
\end{align}
From
\eqref{pprecqqpre}
it follows that $p_u\prec (q_u)^\dagger$ for every $1\leq u\leq s+1$.
 the monomial
\[a_{\mathsf f(p_{1}^{})}^{}\cdots
a_{\mathsf f(p_{s+1}^{})}^{}
\]
appears in \eqref{Ssub1s++1} exactly once. Therefore the right hand side of 
\eqref{Ssub1s++1} can be expressed as
\[
\psi(\bfx)
\sfh_{\xi_\bfx}\left(
\partial_{{(p_1)^\dagger,q_1^{}}}
\cdots
\partial_{{(p_{s+1})^\dagger,q_{s+1}^{}}}
\mathbf d\right),
\]
where 
$\psi\in\sP(\g a)$ is a nonzero polynomial. It follows that 
the set consisting of all $\bfx\in\g a$ which satisfy $\psi(\bfx)\neq 0$ is a Zariski dense subset of $\g a$.
Consequently, the set
\[
\left\{
\bfx\in\C^{m+n}\ :\ 
\sfh_{\xi_\bfx}\left(
\partial_{{(p_1)^\dagger,q_1^{}}}
\cdots
\partial_{{(p_{s+1})^\dagger,q_{s+1}^{}}}
\mathbf d\right)=0
\right\}
\]
is both Zariski dense and Zariski closed. 
This completes the proof of \eqref{hvcthht}.
\end{proof}
We are now ready to prove Proposition \ref{DGBRVVV}.

\begin{proof}
By the equality $\sfh_{\bfx}\circ\iota_\g a^*=\sfh_{\xi_\bfx}$,
it suffices to show that
for every $\mathbf d\in\sP(W)^\g k$,
if
$\sfh_{\xi_{\bfx}}(\mathbf d)= 0$
 for every 
$\bfx\in\g a$,
 then 
$\mathbf d=0$.
Let $\mathscr D\sseq\sP(W)$ be the subalgebra 
generated by 
\[
\Big\{y_{k,k}\,:\,1\leq k\leq m
\Big
\}
\cup
\Big\{y_{\oline{2l-1},\oline{2l}}\,:\,1\leq l\leq n
\Big\}.
\]
Then
$
\mathbf d=\sum_{S\subseteq\mathcal I''}a_S^{}y_S^{}
$,
where $y_S^{}:=\prod_{(i,j)\in S}y_{i,j}$ and 
$a_S^{}\in\mathscr D$ for every $S\subseteq \mathcal I''$. 
Now we fix $S\subseteq \mathcal I''$ and set
$\widetilde\partial:=\prod_{(i,j)\in S}\partial_{{i,j}}
$, so that $\widetilde\partial(\mathbf d)=za_S^{}$ for some scalar $z\neq 0$. By \eqref{hvcthht}, 
\begin{equation}
\label{hthetAS}
\sfh_{\xi_\bfx}(a_S^{})=
\frac{1}{z}\sfh_{\xi_\bfx}
\left(\widetilde\partial(\mathbf d)\right)=0\text{ for every }
\bfx\in\g a.
\end{equation}
From 
\eqref{hthetAS} it follows that $a_S^{}=0$. Since $S\subseteq \mathcal I''$ is arbitrary, we obtain $\mathbf d=0$.
\end{proof}

\section{The Capelli problem for
$\g{gl}(V)\times\g{gl}(V)$ acting on $V\otimes V^*$}
\label{appxB}
In this appendix, we show that the main results of Sections 
\ref{prfof571}--\ref{SecRelSer}, including the abstract Capelli theorem and the
relation between the eigenvalue polynomials $c_\lambda$ and 
the shifted super Jack polynomials of 
\cite[Eq. (31)]{SerVes}, extend to 
the case of 
$\gl(V)\times \gl(V)$ acting on $W:=V\otimes V^*$, where $V:=\C^{m|n}$. These extensions can be  proved along the same lines. However, they can also be deduced from the results of 
\cite{Molev}, and we sketch the necessary arguments below.

Recall the triangular decomposition 
$\gl(m|n)=\g n^-\oplus\g h\oplus\g n^+$  from Section \ref{prfof571}.
Set 
\begin{equation}
\label{canoglgl}
\g g:=\gl(V)\times\gl(V)\cong \gl(m|n)\times \gl(m|n).
\end{equation}
Recall that $\mathrm{H}_{m,n}$
is the set of 
$(m,n)$-hook partitions, as in Definition
\ref{dfnhookprn}. 
We define a map 
$\boldsymbol\Gamma:\mathrm{H}_{m,n}\to\g h^*$ by
$\boldsymbol\Gamma(\flat):=
\sum_{k=1}^m\flat_k\eps_k+\sum_{l=1}^n
\flat_{ l}^*\eps_{\oline l}
$,
where
$\flat_l^*:=\max\{\flat'_l-m,0\}$.
Set 
$\breve{\mathrm{E}}_{m,n,d}:=\boldsymbol\Gamma(\mathrm{H}_{m,n,d})$, 
and let
$
\breve{\mathrm{E}}_{m,n}
:=\bigcup_{d=0}^\infty
\breve{\mathrm{E}}_{m,n,d}$.
For every $\flat\in{\mathrm{H}}_{m,n}$,
let $s_\flat^{}$
denote the \emph{supersymmetric Schur 
polynomial} defined in \cite[Eq. (0.2)]{Molev}, and let
$s_\flat^*$ denote the 
\emph{shifted supersymmetric Schur polynomial}
defined in \cite[Sec. 7]{Molev}.

We can decompose $\sP(W)$ and $\sS(W)$ into irreducible $\g g$-modules, that is, 
\begin{equation}
\label{adv-PW}
\sP^d(W)\cong
\bigoplus_{\mu\in\breve{\mathrm{E}}_{m,n,d}}
V_\mu^*\otimes V_\mu^{}
\,\text{ and }\,
\sS^d(W)\cong
\bigoplus_{
\mu\in\breve{\mathrm{E}}_{m,n,d}}
V_\mu^{}\otimes V_\mu^{*},
\end{equation}
where $V_\mu$ is the irreducible  $\gl(V)$-module of highest weight $\mu$, and $V_\mu^*$
is the contragredient of $V_\mu$.
An argument similar to the proof of
Lemma \ref{prp-g-inv-sPWW}, and based on a decomposition similar to 
\eqref{spdgCB},
implies that
\[
\sPD^d(W)^\g g
\cong
\bigoplus_{k=0}^d
\left(\sP^k(W)\otimes \sS^k(W)\right)^\g g
\cong
\bigoplus_{k=0}^d\,\bigoplus_{\lambda\in \breve{\mathrm{E}}_{m,n,k}}\C D_\lambda,
\]
where $D_\lambda$ is the $\g g$-invariant differential operator that corresponds   to the identity map in
$\Hom_{\gl(V)\times\gl(V)}^{}(V_\lambda^*\otimes V_\lambda^{},V_\lambda^{*}\otimes V_\lambda^{})$.
As in \eqref{ppooll}, the $\g g$-action on
$\sP(W)$ can be realized by polarization operators, and therefore we obtain a Lie superalgebra homomorphism 
$\g g\to\sPD(W)$. The latter map 
extends to a homomorphism of associative superalgebras 
\begin{equation}
\label{rhocheckII}
\check\rho:\bfU(\g g)\to \sPD(W).
\end{equation} 
There exists a canonical tensor product decomposition
$
\bfU(\g g)\cong\bfU(\gl(V))\otimes \bfU(\gl(V))
$
corresponding to \eqref{canoglgl}.
Set 
$
\bfZ^d(\gl(V)):=\bfZ(\gl(V))\cap\bfU^d(\gl(V))
$
for every integer $d\geq 0$.
The next theorem extends Theorem
\ref{prpgw}. 
\begin{thm}
\label{thabsscap}
\emph{(Abstract Capelli Theorem for $W:=V\otimes V^*$.)}
The restrictions
\[
\bfZ^d(\gl(V))\otimes 1\to\sPD^d(W)^{\g g}
\text
{ and }
1\otimes\bfZ^d(\gl(V))\to\sPD^d(W)^{\g g}
\]
of the map $\check\rho$ given in \eqref{rhocheckII}
are surjective.
\end{thm}
\begin{proof}
The statement is a consequence of the results of \cite{Molev}.
In \cite[Thm 7.5]{Molev}, a family of elements $\mathbb S_\flat\in\bfZ^{|\flat|}(\gl(V))$ is constructed which is parametrized by partitions $\flat\in\mathrm{H}_{m,n}$.
Furthermore, in \cite[Thm 8.1]{Molev}
it is proved that
the map 
$\bfZ^{|\flat|}(\gl(V))\otimes 1\to\sPD^{|\flat|}(W)^\g g$
takes $\mathbb S_\flat$ to a differential operator $\Delta_\flat\in\sPD^{|\flat|}(W)^{\g g}$, which is defined in \cite[Sec. 8]{Molev}. 
To complete the proof of surjectivity of the map $\bfZ^d(\gl(V))\otimes 1\to\sPD^d(W)^{\g g}$, it is enough to show that for every $d\geq 0$,  the sets 
\[
\big\{\Delta_\flat\,:\,\flat\in\mathrm{H}_{m,n,k}, 0\leq k\leq d\big\}\text{ and }
\big\{D_\lambda\,:\,\lambda\in\breve{\mathrm{E}}_{m,n,k},0\leq k\leq d\big\}
\]
span the same subspace of $\sPD^d(W)^\g g$.
Since the above two sets have an equal number of elements, it is enough to show that the elements of 
$\big\{\Delta_\flat\,:\,\flat\in\mathrm{H}_{m,n,k}, 0\leq k\leq d\big\}$
are linearly independent. To prove the latter statement, we note that the spectrum of $\Delta_\flat$ 
can be expressed  in terms of the Harish-Chandra image of $\mathbb S_\flat$, which 
by \cite[Thm 7.5]{Molev} is equal to the shifted supersymmetric Schur polynomial $s_\flat^*$. Since the $s_\flat^*$ are linearly independent (see \cite[Cor. 7.2]{Molev} and subsequent remarks therein), the operators $\Delta_\flat$ are also linearly independent.
%It can also be proved by an argument similar to the proof of Theorem \ref{prpgw}.
\end{proof}

%The Capelli basis $\{D_\lambda\,:\,{\lambda\in\breve{\mathrm{E}}_{m,n}}\}
%$ 
%of $\sPD(W)^{\g g}$ is defined as in
%Definition \ref{DefDla}, that is, $D_\lambda$ is the 
%$\g g$-invariant differential operator corresponding to $1_{W_\lambda}\in\mathrm{End}_\C(W_\lambda)$. 

For $\mu\in\breve{\mathrm{E}}_{m,n}$,
set $W_\mu:=V_\mu^*\otimes V_\mu$.
The operator 
$D_\lambda$ acts on $W_\mu$ by a 
scalar
$c_\lambda(\mu)\in\C$. For $\flat\in\mathrm{H}_{m,n}$, let 
$\breve{H}(\flat)$ denote the product of 
the hook lengths of all of the boxes in the Young diagram representation of $\flat$. We define a map
\[
\breve{\boldsymbol\Gamma}_\circ:
\breve{\mathrm{E}}_{m,n}\to
\C^{m+n}
\ \,,\ \,
\sum_{k=1}^m\flat_k\eps_k+\sum_{l=1}^n\flat_l^*\eps_{\oline l}\mapsto
(\flat_1,\ldots,\flat_m,\flat_1^*,\ldots,\flat_n^*).
\]
As in Section \ref{Sec-Sec5}, we denote the homogeneous part of highest degree of 
any
$p\in\sP(\g h^*)$ by $\oline p$. 
\begin{thm}
\label{THMAppB2}
Let $\lambda\in\breve{\mathrm{E}}_{m,n,d}$
and let $\flat\in\mathrm{H}_{m,n,d}$ be such that
$\boldsymbol\Gamma(\flat)=\lambda$. Then \[
c_\lambda=\frac{d!}{\breve{H}(\flat)}s_\flat^*
\circ\breve{\boldsymbol\Gamma}_\circ
.
\]
 In particular, $\oline c_\lambda=
\frac{d!}{\breve{H}(\flat)} 
 s_\flat^{}\circ \breve{\boldsymbol\Gamma}_\circ$.
\end{thm}
\begin{proof}
Follows from \cite[Thm 7.3]{Molev} and 
\cite[Thm 7.5]{Molev}.
\end{proof}

The Frobenius map 
\[
\mathsf{F}:\mathrm{H}_{m,n}\to \C^{m+n}
\] 
of \cite[Sec. 6]{SerVes} is given by
$\mathsf{F}(\flat):=(\sfx(\flat),\ldots,\sfx_m(\flat),\sfy_1(\flat),\cdots,\sfy_n(\flat))$, where
\[
\begin{cases}
\sfx_k(\flat):=\flat_k-k+\frac12(1-n+m)
 & \text{ for }1\leq k\leq m, \\ 
\sfy_l(\flat):=\flat_l^*-l+\frac12(m+n+1) & \text{ for }1\leq l\leq n.%
\end{cases}
\]
As in Lemma \ref{lem6..3}, the map
$\boldsymbol\Gamma\circ\mathsf{F}^{-1}$ extends to an affine 
linear map $\Psi:\C^{m+n}\to \g h^*$. 
Thus we can define the Frobenius transform \[
\mathscr{F}:\sP(\g h^*)\to\sP(\C^{m+n})
\] as in Definition 
\ref{FrobT}, namely by $\mathscr{F}(p):=p\circ\Psi$.

%For $\mu\in\breve{\mathrm{E}}_{m,n}$, 

For $\lambda\in\breve{\mathrm{E}}_{m,n,d}$, 
by Theorem \ref{thabsscap} there exists an element
 $z_\lambda\in \bfZ(\gl(V))\cap\bfU^d(\gl(V))$
 such that 
$D_\lambda=\check\rho(1\otimes z_\lambda)$. 
Then for a highest weight vector
$v_\mu^*\otimes v_\mu^{}\in V_\mu^*\otimes V_\mu^{}\cong W_\mu$, we have
\[
c_\lambda(\mu)v_\mu^*\otimes v_\mu^{}
=
D_\lambda
v_\mu^*\otimes v_\mu^{}
=
v_\mu^*\otimes\check\rho(z_\lambda)v_\mu^{}
=\mu(\HC^+(z_\lambda))v_\mu^*\otimes v_\mu^{},
\]
where $\HC^+:\bfU(\gl(V))\to \bfU(\g h)$ is the Harish-Chandra projection corresponding to the decomposition
$
\bfU(\gl(V))=
\left(\bfU(\gl(V))
\g n^+
+
\g n^-\bfU(\gl(V))
\right)
\oplus\bfU(\g h)
$.
From the description of the image of the Harish-Chandra projection (see for example \cite{sergeev82}, \cite{sergeev3}, \cite{KacZ}, \cite{Gorelik}, or \cite[Sec. 2.2.3]{ChWabook}), we obtain 
explicit generators for $\HC^+(\bfZ(\gl(V))$, similar to the $G_d$ defined in \eqref{dfnGd}.
Furthermore, by the canonical isomorphism
$\bfU(\g h)\cong\sS(\g h)\cong \sP(\g h^*)$, we can 
consider 
$\HC^+(\bfZ(\gl(V))$ as a subalgebra of $\sP(\g h^*)$, which we henceforth denote by 
$\mathbf{I}(\g h^*)$.
A direct calculation using the explicit 
generators of $\mathbf{I}(\g h^*)$ proves that $\mathscr{F}(\mathbf I(\g h^*))\subseteq\Lambda^\natural_{m,n,1}$, where 
$\Lambda^\natural_{m,n,1}$ 
consists of polynomials $f(\sfx_1,\ldots,\sfx_m,\sfy_1,\ldots,\sfy_n)$, 
which are 
separately 
symmetric in $\sfx:=(\sfx_1,\ldots,\sfx_m)$ and
in $\sfy:=(\sfy_1,\ldots,\sfy_n)$, and which satisfy
the relation 
\begin{equation*}
\textstyle f(\sfx+\frac{1}{2}\mathsf{e}_{k,m},\sfy-\frac{1}{2}\mathsf{e}_{l,n})=f(\sfx-\frac{1}{2}\mathsf{e}_{k,m},\sfy+\frac{1}{2}\mathsf{e}_{l,n}) 
\end{equation*}
on every hyperplane $\sfx_k+ \sfy_l=0$, where $1\leq k\leq m$ and $1\leq
l\leq n$. Now let $SP_\flat^*$ be the basis of 
$\Lambda^\natural_{m,n,1}$ introduced in \cite[Sec. 6]{SerVes} for $\theta=1$.
An argument similar to the proof of 
Theorem \ref{thmconnSV} 
yields the following statement.
\begin{thm}
\label{THMAB3}
Let
$\lambda\in\breve{\mathrm{E}}_{m,n,d}$, 
and let $\flat\in\mathrm{H}_{m,n,d}$ be chosen such that
$\lambda=\boldsymbol\Gamma(\flat)$, then
\[
\mathscr{F}(c_\lambda)=\frac{d!}{\breve{H}(\flat)}SP_\flat^*.
\]

\end{thm}
%Finally, from the results of 
%\cite[Sec. 7]{Molev}
%it follows that the homogeneous part of highest degree in the polynomial $c_\lambda$ 
%is equal to a supersymmetric Schur function, which is indeed
%the spherical vector for the group-like symmetric superpair 
%$(\gl(V)\times\gl(V),\gl(V))$.

\end{alphasection}

\end{document}